\documentclass[10pt]{amsart}
\usepackage{amsmath}
\usepackage{amssymb}
\usepackage[margin = 1in]{geometry}
\usepackage{mathtools}
\usepackage{appendix}
\usepackage{enumitem}
\usepackage{tikz}
\usetikzlibrary{cd}
\usepackage{mathrsfs}
\usepackage[normalem]{ulem}
\usepackage{accents}
\usepackage[new]{old-arrows}
\usepackage{tikz-3dplot}
\usepackage{dsfont}
\usepackage{pgfplots}
\usepgfplotslibrary{polar}
\pgfplotsset{compat=newest}
\usepackage{hyperref}
\usepackage{tabularray}
\usepackage{graphicx}
\usepackage{subfig}
\usepackage{blkarray}
\usepackage{todonotes}
\usepackage{setspace}
\usepackage{changes}
\usetikzlibrary{arrows,decorations.markings}

\DeclarePairedDelimiter\Par{(}{)}

\DeclarePairedDelimiter\BBrc{\{}{\}}

\DeclareMathOperator{\Char}{char}

\DeclareMathOperator\Div{Div}

\DeclareMathOperator\Pic{Pic}

\DeclareMathOperator\Bl{Bl}

\DeclareMathOperator\Spec{Spec}

\DeclareMathOperator\Ind{Ind}

\DeclareMathOperator\Bir{Bir}
\DeclareMathOperator\id{id}
\DeclareMathOperator\mult{mult}

\newcommand{\ba}{\mathbb{A}}

\newcommand{\bc}{\mathbb{C}}

\newcommand{\bp}{\mathbb{P}}
\newcommand{\bq}{\mathbb{Q}}
\newcommand{\br}{\mathbb{R}}

\newcommand{\bz}{\mathbb{Z}}

\newcommand{\cb}{\mathcal{B}}
\newcommand{\cc}{\mathcal{C}}

\newcommand{\co}{\mathcal{O}}

\newcommand{\fb}{\mathfrak{b}}
\newcommand{\fc}{\mathfrak{c}}
\newcommand{\ff}{\mathfrak{f}}

\newcommand{\fh}{\mathfrak{h}}

\newcommand{\fl}{\mathfrak{l}}

\newcommand{\fq}{\mathfrak{q}}

\newcommand{\tX}{\widetilde{X}}

\DeclareMathOperator{\Bb}{\Bl_{\mathcal{B}}}
\DeclareMathOperator{\Bc}{\Bl_{\mathcal{C}}}

\newcommand{\tphi}{\tilde{\varphi}}
\newcommand{\tpsi}{\tilde{\psi}}
\newcommand{\bphi}{\bar{\varphi}}
\newcommand{\bpsi}{\bar{\psi}}
\newcommand{\bchi}{\bar{\chi}}
\newcommand{\tPhi}{\tilde{\Phi}}
\newcommand{\tPsi}{\tilde{\Psi}}
\newcommand{\wtPhi}{\widetilde{\Phi}}
\newcommand{\wtPsi}{\widetilde{\Psi}}

\title[Degeneration of surface automorphisms]{The degeneration of a family of rational surface automorphisms}
\author{Qitong Jiang}
\address{Department of Mathematics, The Pennsylvania State University, University Park, PA 16802, USA}
\email{qvj5031@psu.edu}

\date{\today}

\newtheorem{thm}{Theorem}
\newtheorem{prop}[thm]{Proposition}
\newtheorem{cor}[thm]{Corollary}
\newtheorem{lemma}[thm]{Lemma}

\theoremstyle{definition}

\newtheorem{example}{Example}

\numberwithin{equation}{section}
\makeatletter
\patchcmd{\@verbatim}
  {\verbatim@font}
  {\verbatim@font\scriptsize}
  {}{}
\makeatother

\begin{document}

\begin{abstract}
We consider a one-dimensional family of rational surfaces with automorphisms discussed in \cite{Gizatullin1994,Blanc2008,Blanc2013}. In a degeneration of this family, the limiting map is the identity map on a special fiber. We check that the map on the total space of the family has indeterminacy in the special fiber. However, we show that after blowing-up at an indeterminate curve, there is an induced birational map on the exceptional divisor over the indeterminate curve. Moreover, we show that this map has dynamical degree $\lambda=16$.
\end{abstract}

\maketitle

\setcounter{secnumdepth}{2}
\setcounter{tocdepth}{2}

\section{Introduction}

Given an automorphism $f:X\to X$ of an algebraic variety $X$, if we blow-up $X$ at a subvariety $V$ with $f(V)=V$, then $f$ induces an automorphism $\tilde{f}\vert_E:E\to E$ of the exceptional divisor $E$. This map does not have any new dynamical properties, as it fibers over the original map $f\vert_V:V\to V$. On the other hand, if $f(V)=V'\neq V$, then the lift $\tilde{f}$ on $\Bl_V{X}$ simply contracts $E$.

In this paper, we will be concerned with the situation when $f:X\dashrightarrow X$ is a birational map and $V$ is contained in $\Ind(f)$. In this case, interesting dynamics are possible on $E\subset\Bl_V(X)$. In particular, we focus on cases related to degenerations of rational surface automorphisms. Let $T$ be an elliptic curve and $X=\bp^2\times T$. Given $t\in T$, there is an involution $\iota_t$ on $\bp^2\cong\bp^2\times\BBrc{t}$ associated with $t$ and $T$ constructed in \cite{Gizatullin1994,Blanc2008,Blanc2013}. Fix a non-trivial $2$-torsion point $q\in T$. We have the maps $\varphi,\psi,\chi:X\dashrightarrow X$ defined by
\begin{equation*}
	\varphi(p,t)=(\iota_q(p),t),\quad\psi(p,t)=(\iota_t(p),t),\quad\chi(p,t)=(\iota_{-t}(p),t),
\end{equation*}
where $-t$ is the inverse of $t$. The map $\vartheta=\varphi\circ\chi\circ\varphi\circ\psi$ can be viewed as a family of birational maps of $\bp^2$:
\begin{equation*}
	\BBrc{\iota_q\circ\iota_{-t}\circ\iota_q\circ\iota_t:t\in T}.
\end{equation*}
For a general $t\in T$, the map $\vartheta_t:X_t\dashrightarrow X_t$, the restriction of $\vartheta$ to the fiber $X_t=\bp^2\times\BBrc{t}$ over $t$, can be regularized to an automorphism on a higher birational model of $X_t$, and the dynamical degree of $\vartheta_t$ is $\lambda(\vartheta_t)=17+12\sqrt{2}$, as computed from its action on divisors (cf. Lemma 3.3 in \cite{Blanc2013}; see also \cite{Bot2021}); whereas the dynamical degree of $\vartheta_q$ is $\lambda(\vartheta_q)=1$ as $\vartheta_q=\iota_q^4=\id_{\bp^2}$. As the family degenerates to the identity on the fiber $X_q=\bp^2\times\BBrc{q}$, $\vartheta$ develops indeterminate curves inside $X_q$. In fact, the interesting dynamics do not completely disappear as the family degenerates; rather, they become concentrated along the indeterminate curves. To see this, we blow-up $X$ at the indeterminate curves and find that the induced map on the exceptional divisor has dynamical degree $\lambda=16$.

Here is the outline of the discussion:
\begin{itemize}
	\item In Section \ref{sec:3}, we exhibit families of birational maps of rational surfaces whose general member can be regularized as an automorphism. We study the indeterminate curves in the fiber over $q$ of the corresponding threefold map by factoring the map as a sequence of automorphisms and Atiyah flops.
	\item In Section \ref{sec:4}, we find induced self-maps of the exceptional divisors over the identified indeterminate curves in the special fiber over $q$.
	\item In Section \ref{sec:5}, we study the dynamical properties of the induced self-maps of the exceptional divisors.
\end{itemize}

In particular, we show the following result.
\begin{thm}
\label{mainthm}
There exists a curve $\fq\subset X_q$ which is indeterminate for the map $\vartheta=\varphi\circ\chi\circ\varphi\circ\psi$. Let $\Bl_\fq{X}$ be the blow-up of $X=\bp^2\times T$ at $\fq$ with exceptional divisor $E_\fq$. There is a birational self-map of $E_\fq$ induced by $\vartheta$, and the map has dynamical degree $\lambda=16$.
\end{thm}

As a warm-up, let us consider several examples of birational self-maps of $\bp^3$ where the induced maps on exceptional divisors can be easily obtained. In particular, we blow-up $\bp^3$ at the line $L$ defined by $y=z=0$ to get $\pi:\Bl_L\bp^3\to\bp^3$. Then we consider the lifts of several birational maps $f:\bp^3\dashrightarrow\bp^3$, i.e., $\tilde{f}=\pi^{-1}\circ f\circ\pi$, and the induced maps $\tilde{f}\vert_E$ on the exceptional divisor $E$ over $L$. The exceptional divisor $E$ is isomorphic to $\bp^1\times\bp^1$ on $\Bl_{L}\bp^3$.

\begin{example}
\label{example:dilation}
Given $b\in\bc$, let $f:\bp^3\to\bp^3$ be the automorphism defined by the formula
\begin{equation*}
	[w:x:y:z]\mapsto[bw:x:y:z].
\end{equation*}
Then $f(L)=L$. The induced map $\tilde{f}\vert_E:E\to E$ can be written as
\begin{equation*}
	([u:v],[s:t])\mapsto([bu:v],[s:t]).
\end{equation*}
In the affine chart $\ba^1\times\ba^1$ where $v,t\neq 0$, this is a dilation in the $u$-diration:
\begin{equation*}
	(u,s)\mapsto(bu,s).
\end{equation*}
\end{example}
\begin{example}
\label{example:bending}
For a slightly more complicated example, let $f:\bp^3\dashrightarrow\bp^3$ be the birational map given by
\begin{equation*}
	[w:x:y:z]\mapsto[x^2w:x^3:x^2y:x^2y-ayw^2+x^2z],
\end{equation*}
for some $a\in\bc$. Then $f$ is defined on a dense open subset of the line $L$. The induced map $\tilde{f}\vert_E:E\dashrightarrow E$ is birational in this case:
\begin{equation*}
	([u:v],[s:t])\mapsto([u:v],[v^2t-au^2t+v^2s:v^2t]).
\end{equation*}
We note that in the affine chart where $v,t\neq 0$, this is an ``area-preserving'' bend:
\begin{equation*}
	(u,s)\mapsto(u,1-au^2+s).
\end{equation*}
\end{example}

\begin{example}
In general, given a birational map $f:X\dashrightarrow X$, if $V$ is a subvariety of codimension $\geq 2$ and $f(V)\neq V$, then the lifted map on $\Bl_V{X}$ simply contracts $E$. Let $f:\bp^3\to\bp^3$ be the automorphism
\begin{equation*}
	[w:x:y:z]\mapsto[y:z:w:x].
\end{equation*}
Then the image of the line $L$ under $f$ is not $L$. The lifted map $\tilde{f}:\Bl_L\bp^3\dashrightarrow\Bl_L\bp^3$ will contract the exceptional surface $E$ to the strict transform of the line $\BBrc{w=x=0}\subset\bp^3$.
\end{example}

\begin{example}
\label{example:reflection}
Suppose that $V$ is a subvariety of codimension $\geq 2$ that lies in $\Ind(f)$, and we blow-up $\bp^3$ at $V$. We consider when there is an induced birational map on the exceptional divisor. For example, let $f:\bp^3\dashrightarrow\bp^3$ be the composition of the automorphism:
\begin{equation*}
	[w:x:y:z]\mapsto[y:z:w:x]
\end{equation*}
and the standard involution of $\bp^3$:
\begin{equation*}
	[w:x:y:z]\mapsto[xyz:wyz:wxz:wxy].
\end{equation*}
The line $L=\BBrc{y=z=0}$ is contained in the indeterminacy locus of $f$. In this case, $f$ induces an automorphism $\tilde{f}\vert_E:E\to E$, whose formula can be written as
\begin{equation*}
	([u:v],[s:t])\mapsto([s:t],[u:v]).
\end{equation*}
We note that, in the affine chart where $v,t\neq 0$, this map is the reflection across the line $u=s$:
\begin{equation*}
	(u,s)\mapsto(s,u).
\end{equation*}
\end{example}

\begin{example}
\label{example:henon}
We can actually obtain a Hénon map on the exceptional divisor $E$ by composing the maps introduced in Examples \ref{example:dilation}, \ref{example:bending} and \ref{example:reflection}. The composition $f$ is undefined along the line $L=\BBrc{y=z=0}$, but the induced map on the exceptional divisor $E$ over $L$ is a Hénon map.
\end{example}

For our main discussion, we recall some preliminaries.

\section{Groundwork}

We assume that the base field $K$ is algebraically closed of characteristic $\Char{K}=0$.

\subsection{Relative numerical equivalence}

Let $X$ be a smooth variety over $K$, and let $p:X\to S$ be a proper morphism onto a variety $S$. We define $Z_1(X/S)$ to be the free abelian group generated by reduced irreducible curves which are mapped to points by $p$. There is a numerical equivalence $\equiv_S$ in $\Pic(X)$ defined by
\begin{equation*}
	D\equiv_S D'\quad\text{if and only if}\quad D\cdot C=D'\cdot C\text{ for all }C\in Z_1(X/S).
\end{equation*}
We write $N^1(X/S):=(\Pic(X)/\equiv_S)\otimes_\bz\br$. When $S=\Spec{K}$, we drop $/S$ from the notation.

\subsection{Blow-up and exceptional locus}

The blow-up of a variety $X$ at a subvariety $V$ will be written as $\Bl_VX$. Let $\pi:\Bl_VX\to X$ be the projection map. The exceptional divisor $\pi^{-1}(V)$ is isomorphic to $\bp(N_{V/X})$, where $N_{V/X}$ is the normal bundle to $V$ in $X$. We denote $\pi^{-1}(V)$ by $E_V$. We can thus identify the points on $E_V$ with elements in $\bp(N_{V/X})$.

We will frequently consider sequences of blow-ups throughout the paper. To simplify notations, we use the same notation for a subvariety $V$ and all of its strict transforms.

\subsection{Dynamical degree and algebraic stability}

Let $X$ be a smooth projective variety of dimension $k$, and let  $f:X\dashrightarrow X$ be a birational map. There exists a smooth projective variety $\tX$ and proper morphisms $\pi:\tX\to X$ and $\sigma:\tX\to X$ such that the diagram commutes:
\begin{equation*}
\begin{tikzcd}
	& \widetilde{X} \arrow[dl,"\pi"'] \arrow[dr,"\sigma"] & \\
	X \arrow[rr,dashed,"f"] & & X.
\end{tikzcd}
\end{equation*}
If $D\in\Div(X)$ is a divisor, then we define its \textit{\textbf{pullback}} (under $f$) by $f^*(D):=\pi_*\sigma^*(D)$ and its \textit{\textbf{pushforward}} by $f_*(D):=\sigma_*\pi^*(D)$. Now let $H\in\Div(X)$ be an ample divisor. We define the \textit{\textbf{($H$-)degree}} of $f$ to be
\begin{equation*}
	\deg_H(f):=H^{k-1}\cdot f^*(H).
\end{equation*}
The \textit{\textbf{(first) dynamical degree}} of $f:X\dashrightarrow X$ is defined to be
\begin{equation*}
	\lambda(f):=\lim_{n\to\infty}\Par*{\deg_H(f^n)}^{1/n}.
\end{equation*}
The fact that the above limit exists and is independent of the choices of $H$ is proved in \cite{Dang2020,Truong2020}.

We focus on the case when $X$ is a smooth projective surface. In this case, we have that $(f^{-1})_*=f^*$ and $(f^{-1})^*=f_*$ and that
\begin{equation}
\label{eqn:degreeff-1}
	\deg_H(f)=H\cdot f^*(H)=f_*(H)\cdot H=\deg_H(f^{-1}).
\end{equation}
So $\lambda(f)$ can also be defined using pushforward and is the same as $\lambda(f^{-1})$. Let us also consider the induced linear maps
\begin{equation*}
	f^*:N^1(X)\to N^1(X)\quad\text{and}\quad f_*:N^1(X)\to N^1(X).
\end{equation*}
The birtional map $f:X\dashrightarrow X$ is said to be \textit{\textbf{algebraically stable}} if
\begin{equation*}
	(f^{n})^*=(f^*)^n\quad(\text{or equivalently }(f^n)_*=(f_*)^n)
\end{equation*}
for any integer $n\geq 0$. The following theorem is useful in practice.

\begin{thm}[Proposition 1.13 and Theorem 1.14 in \cite{DF2001}]
\label{thm:AS}
Let $S$ be a smooth projective surface, and let $f,g:S\dashrightarrow S$ be birational maps. Then
\begin{equation*}
	(g\circ f)_*=g_*f_*
\end{equation*}
if and only if there is no irreducible curve $C\subset S$ such that $f(C)\subset\Ind(g)$. Moreover, $f:S\dashrightarrow S$ is algebraically stable if and only if there is no irreducible curve $C\subset S$ and $n\geq 0$ such that $f(C)\subset\Ind(f^n)$.
\end{thm}

We use the following theorem to compute the dynamical degree for an algebraically stable birational self-map of a surface.

\begin{thm}[Corollary 1.19 in \cite{DF2001}]
\label{thm:radius}
Let $S$ be a smooth projective surface, and let $f:S\dashrightarrow S$ be an algebraically stable birational map. Then $\lambda(f)$ coincides with the spectral radius of $f_*:N^1(S)\to N^1(S)$.
\end{thm}

\subsection{Atiyah flop}

Let $Y$ be a smooth quasi-projective threefold. A curve $\fb\subset Y$ is called a \textit{\textbf{$(-1,-1)$-curve}} if $\fb\cong\bp^1$ and its normal bundles satisfies
\begin{equation*}
	N_{\fb/Y}\cong\co_{\bp^1}(-1)\oplus\co_{\bp^1}(-1).
\end{equation*}
There is a diagram
\begin{equation*}
\begin{tikzcd}
	& E\subset Z\quad\quad \arrow[dl,"\pi"'] \arrow[dr,"\pi^+"] & \\
	\fb\subset Y \arrow[rr,dashed,"\phi"] & & Y^+\supset\fb^+,
\end{tikzcd}
\end{equation*}
where $\pi$ is the blow-up of $Y$ at $\fb$, $\pi^+$ is the blow-up of $Y^+$ at $\fb^+$, and
$E\cong\bp^1\times\bp^1$ is the common exceptional divisor for $\pi$ and $\pi^+$ with the two rulings corresponding to the two contractions (see \cite{Reid1983}). The birational map $\phi$ obtained in this way is called an \textit{\textbf{Atiyah flop}}. We note that the map $\phi$ is an isomorphism outside the curve $\fb$. The $(-1,-1)$-curve $\fb$ is also called the \textit{\textbf{flopping curve}} of $\phi$, and the curve $\fb^+$ is called the \textit{\textbf{flopped curve}}. (In this paper, we do not consider flops of curves with normal bundle $\co\oplus\co(-2)$.)

Let $X$ be a quasi-projective threefold, and let $B$ and $C$ be two curves on $X$ intersecting only at a point $p\in X$ with distinct tangent directions. Let
\begin{align*}
	\tX:&=\Bl_C\Bl_BX\overset{\pi_2}{\longrightarrow}\Bl_BX\overset{\pi_1}{\longrightarrow}X\\
	\tX^+:&=\Bl_B\Bl_CX\overset{\pi_2^+}{\longrightarrow}\Bl_CX\overset{\pi_1^+}{\longrightarrow}X
\end{align*}
be the blow-ups of $X$ at $B$ and $C$ in different orders. Lastly, we let $b$ be the preimage of the point $p$ on $\Bl_B{X}$, which is a curve isomorphic to $\bp^1$; and we let $\fb$ be the strict transform of the curve on $\tX$.

\begin{prop}
\label{prop:atiyahflop}
The normal bundle $N_{\fb/\tX}$ is isomorphic to $\co_{\bp^1}(-1)\oplus\co_{\bp^1}(-1)$, and the natural birational map $\phi:\tX\dashrightarrow \tX^+$ is an Atiyah flop with the flopping curve $\fb$.
\end{prop}
\begin{proof}
On $\widehat{X}:=\Bl_B{X}$, the curve $b$ is a fiber of the exceptional divisor $E_B$. So the self-intersection number of $b$ as a curve on the surface $E_B$ is $0$. The total transform of $E_B$ on $\tX$ is irreducible, so it is also the strict transform of $E_B$. In fact, it is isomorphic to the blow-up of $E_B\subset\widehat{X}$ at a point on the fiber $b$. Therefore, the strict transform $\fb$ of $b$, considered as a curve on the surface $E_B\subset\tX$ (the strict transform of $E_B$), has self-intersection number $-1$. If we can show that $K_{\tX}\cdot\fb=0$, by the Remark 5.2 (a) in \cite{Reid1983}, there is an exact sequence
\begin{equation*}
	0\longrightarrow\co_{\bp^1}(-1)\longrightarrow N_{\fb/\tX}\longrightarrow\co_{\bp^1}(-1)\longrightarrow 0;
\end{equation*}
and this shows that $N_{\fb/\tX}\cong\co(-1)\oplus\co(-1)$.

The canonical divisor of $\tX$ is
\begin{equation*}
	K_{\tX}=\pi_2^*\pi_1^*K_X+\pi_2^*E_B+E_C
\end{equation*}
By the push-pull formula, $\pi_2^*\pi_1^*K_X$ is trivial on $\fb$, and $\pi_2^*E_B\cdot\fb=E_B\cdot b$. Moreover, since the curve $\fb$ intersects $E_C$ transversally at a point, we have $E_C\cdot\fb=1$. To compute the intersection number $E_B\cdot b$, we note that
\begin{equation*}
	E_B\cdot b=\deg(\co_{\widehat{X}}(E_B)\vert_b)=\deg(\co_{\widehat{X}}(E_B)\otimes\co_{E_B}\vert_b)=\deg(N_{E_B/\widehat{X}}\vert_b).
\end{equation*}
Since $N_{E_B/\widehat{X}}\cong\co_{E_B}(-1)$, the restriction of $N_{E_B/\widehat{X}}$ to $b\cong\bp^1$ is isomorphic to $\co_{\bp^1}(-1)$. Thus, $E_B\cdot b=-1$. We can therefore conclude that
\begin{equation*}
	K_{\tX}\cdot\fb=0-1+1=0.
\end{equation*}
\end{proof}

In the case where a pseudo-automorphism $f$ of a quasi-projective threefold can be factored into a sequence of Atiyah flops and automorphisms, the image of a divisor class can be computed recursively. We need the following lemma.

\begin{lemma}[Lemma 8 in \cite{Lesieutre2016}]
\label{lemma:flop}
Let $\phi:X\dashrightarrow X^+$ be an Atiyah flop with the flopping curve $c\subset X$. Let $c^+$ denote the flopping curve of $\phi^{-1}$. Suppose that $D^+$ is a divisor on $X^+$ and that $D$ is the pullback of $D^+$ under $\phi$. Then
\begin{equation}
\label{eqn:formula}
    \mult_{c}(D)=\mult_{c^+}(D^+)+D^+\cdot c^+.
\end{equation}
\end{lemma}

Let $p:X\to T$ be a proper and flat family of surfaces parametrized by a curve $T$. Let $\phi:X\dashrightarrow X^+$ be an Atiyah flop such that the flopping curve $c$ lying in the fiber $F:=X_q$ for some $q\in T$ and that the self-intersection number of $c$ in $F$ is $-2$. Let $c^+$ denote the flopped curve, and let $F^+:=\phi(F)\subset X^+$. For a divisor $D^+\in N^1(X^+/T)$, we have
\begin{equation*}
    \phi^*(D^+)\vert_{F}=(\phi\vert_F)^*(D^+\vert_{F^+})+(D^+\cdot c^+)c.
\end{equation*}
Indeed, by Lemma \ref{lemma:flop}, we have
\begin{align*}
    \mult_{c}(\phi^*(D^+)\vert_{F})=\mult_{c}(D)=\mult_{c^+}(D^+)+D^+\cdot c^+=\mult_{c^+}(D^+\vert_{F^+})+D^+\cdot c^+.
\end{align*}
Since the self-intersection number of $c$ in $F$ is $-2$, by the Remark 5.13 (a) in \cite{Reid1983}, $\phi\vert_F:F\to F^+$ is an isomorphism, and so
\begin{equation*}
    \mult_{c}((\phi\vert_F)^*(D^+\vert_{F^+}))=\mult_{c^+}(D^+\vert_{F^+}).
\end{equation*}
The equation (\ref{eqn:formula}) follows.

Inductively, we obtain the following result:
\begin{prop}
\label{prop:formula}
Let $X$ be a threefold, and let $f:X\dashrightarrow X$ be a pseudo-automorphism of $X$ that preserves the fibration $p:X\to T$, with indeterminate curves lying in the fiber $F$. Suppose $f$ factors into a sequence of flops and isomorphisms
\begin{equation*}
    f:X=X_1\overset{\phi_1}{\dashrightarrow}X_1^+\overset{f_1}{\longrightarrow}X_2\overset{\phi_2}{\dashrightarrow}X_2^+\rightarrow\dotsm\dashrightarrow X_{n-1}^+\overset{f_{n-1}}{\longrightarrow}X_{n}=X,
\end{equation*}
where $\phi_i$ are Atiyah flops, and $f_i$ are the isomorphisms. Let $c_{i}$ be the flopping curve of $\phi_i$ with flopped curve $c_i^+$. We write
\begin{equation*}
	g_i:=f_{n-1}\circ\dotsm\circ f_{n-i}:X_{n-i}^+\dashrightarrow X_n, i=1,\dots,n-1;
\end{equation*}
and write
\begin{equation*}
	h_i:=f_{n-i-1}\circ\dotsm\circ\phi_1:X_{1}\dashrightarrow X_{n-i}, i=1,\dots,n-2,\quad\text{and}\quad h_{n-1}:=\id_X.
\end{equation*}
Moreover, we assume that the self-intersection number of $c_i$ is $-2$ on $h_{n-i}(F)$. For a divisor $D\in N^1(X/T)$,
\begin{equation}
\label{eqn:differenceformula}
    f^*(D)\vert_F=\Par*{f\vert_F}^*(D\vert_F)+\sum_{i=1}^{n-1}(g_i^*(D)\cdot c_{n-i}^+)h_i^*(c_{n-i}).
\end{equation}
\end{prop}

In particular, if $D$ is ample and $c_{n-i}^+\not\in\Ind(g_i)$ for any $i$, then the curve $c_{n-i}^+$ is not in the base locus of $g_i^*(D)$ and $g_i^*(D)\cdot c_{n-i}^+\geq 0$ for all $i$. Therefore, the difference
\begin{equation*}
	(f^n)^*(D)\vert_F-(f^n\vert_F)^*(D\vert_F)
\end{equation*}
is an effective divisor supported on flopping curves and their pullbacks under $f\vert_F$.

Finally, we observe that there is a natural birational map between two birational models. Let $B$ and $C$ be two curves in a threefold $X$ intersecting at a point $p$ with distinct tangent directions. There are several ways to blow-up $X$:
\begin{enumerate}
	\item The first sequence of blow-ups (at the obvious centers) is
	\begin{equation*}
	\begin{tikzcd}
		\Bl_C\Bl_B\Bl_p{X} \rar["\pi_C"] & \Bl_B\Bl_p{X} \rar["\pi_B"] & \Bl_p{X} \rar["\pi"] & X
	\end{tikzcd}
	\end{equation*}
	with exceptional divisors $E_p$, $E_B$, $E_C$ over $p,B,C$, respectively.
	\item The second sequence of blow-ups is
	\begin{equation*}
	\begin{tikzcd}
		\Bl_C\Bl_{\ff'}\Bl_B{X} \rar["\pi'_C"] & \Bl_{\ff'}\Bl_B{X} \rar["\pi'"] & \Bl_B{X} \rar["\pi'_B"] & X
	\end{tikzcd}
	\end{equation*}
	with exceptional divisors $E_B'$, $E_{\ff'}'$, $E_C'$ over $B,p,C$, respectively. The blow-up center of $\pi'$ is the fiber $\ff'$ of $E_B'\subset\Bl_BX$ over $p$.
	\item The third sequence of blow-ups is
	\begin{equation*}
	\begin{tikzcd}
		\Bl_{\ff''}\Bl_C\Bl_B{X} \rar["\pi''"] & \Bl_C\Bl_B{X} \rar["\pi''_C"] & \Bl_B{X} \rar["\pi''_B"] & X
	\end{tikzcd}
	\end{equation*}
	with exceptional divisors $E_B''$, $E_C''$, $E_{\ff''}''$ over $B,C,p$, respectively. The blow-up center of $\pi''$ is the flopping curve $\ff''$ of the Atiyah flop $\phi:\Bl_C\Bl_B{X}\dashrightarrow\Bl_B\Bl_C{X}$.
\end{enumerate}

A standard exercise (Exercise 21 in \cite{Kollár2008}) states the following result:
\begin{prop}
Let $p\in L\subset\ba^3$ be a point on a line. Let $f\subset\Bl_L\ba^3$ be the preimage of $p$. The identity map on $\ba^3$ induces an isomorphism $\Bl_f\Bl_L\ba^3\cong\Bl_L\Bl_p\ba^3$.
\end{prop}

This is locally analytically the same as considering the iterated blow-ups of a smooth threefold at a point lying on a curve. In our situation, we have the isomorphism between $\Bl_B\Bl_pX$ and $\Bl_{\ff'}\Bl_BX$, and hence an isomorphism between $\Bl_C\Bl_B\Bl_p{X}$ and $\Bl_C\Bl_{\ff'}\Bl_B{X}$.

The birational map (an Atiyah flop) between $\Bl_{C}\Bl_{\ff'}\Bl_B{X}$ and $\Bl_{\ff''}\Bl_C\Bl_B{X}$ is provided by Proposition \ref{prop:atiyahflop}. Here we blow-up $\Bl_B{X}$ at the two intersecting curves $\ff'$ and $C$ in different orders. We observe that $E_p$ on $\Bl_{C}\Bl_{B}\Bl_p{X}$ is isomorphic to the blow-up of $\bp^2$ at two points. Identifying $\Bl_C\Bl_{\ff'}\Bl_B{X}$ with $\Bl_C\Bl_B\Bl_p{X}$, the flopping curve of the Atiyah flop between $\Bl_{C}\Bl_{\ff'}\Bl_B{X}$ and $\Bl_{\ff''}\Bl_C\Bl_B{X}$ is the curve with self-intersection $-1$ on $E_p$ connecting the two exceptional curves. The restriction of the Atiyah flop to $E_p$ is a map that contracts the flopping curve (Remark 5.13 (a) in \cite{Reid1983}).

We summarize the discussion in the following proposition.

\begin{prop}
\label{prop:linkblowups}
There is an isomorphism (induced by the identity map)
\begin{equation*}
	\Bl_C\Bl_B\Bl_p{X}\overset{\cong}{\longrightarrow}\Bl_C\Bl_{\ff'}\Bl_B{X},
\end{equation*}
which identifies $E_p$ and $E_{\ff'}'$. Furthermore, there is an Atiyah flop (induced by the identity map)
\begin{equation*}
	\Bl_C\Bl_B\Bl_p{X}\cong\Bl_C\Bl_{\ff'}\Bl_B{X}\dashrightarrow\Bl_{\ff''}\Bl_C\Bl_B{X}
\end{equation*}
that sends $E_p$, which is isomorphic to $\bp^2$ blown-up at two points, to $E_{\ff''}''\cong\bp^1\times\bp^1$.
\end{prop}

We have acquired all the necessary tools for further discussion.

\section{Family of Maps}
\label{sec:3}

We present families of rational surface birational self-maps in this section and examine their properties.

\subsection{Cubic involution of \texorpdfstring{$\bp^2$}{}} 

We recall the following construction in \cite{Gizatullin1994,Blanc2008,Blanc2013}: Let $T$ be a smooth cubic curve in $\bp^2$, and let $q\in T$ be a point. For a general point $p\in\bp^2$, the line $l_{q,p}$ connecting $q,p$ intersects $T$ at three points $q,a,b$. There is a unique involution $i_{q,p}$ of $l_{q,p}$ with $i_{q,p}(a)=a$ and $i_{q,p}(b)=b$. Therefore, given a smooth cubic curve $T$ and $q\in T$, we can define a cubic involution $\iota_q:\bp^2\dashrightarrow\bp^2$ by
\begin{equation*}
    p\mapsto\iota_q(p):=i_{q,p}(p).
\end{equation*}

We first note that the cubic involution fixes (a dense open subset of) the cubic curve pointwise by construction. Now, let $q$ be a point in $T$ such that there are four distinct lines $\fl_1,\fl_2,\fl_3,\fl_4$ passing through $q$ and being tangent to $T$ at $q_1,q_2,q_3,q_4$ respectively. The points $q_1,q_2,q_3,q_4$ will be sometimes referred to as the points of tangency (or tangent points) associated with $q$. The indeterminacy locus of $\iota_q$ is
\begin{equation*}
    \Ind(\iota_q)=\BBrc{q,q_1,q_2,q_3,q_4}.
\end{equation*}
In fact, the cubic involution $\iota_q$ contracts the conic curve $\fq$ passing through $q,q_1,q_2,q_3,q_4$ to $q$ and contracts the tangent line $\fl_i$ to $q_i$ for each $i$. The (affine) picture of the cubic curve, the conic curve, and the tangent lines is in Figure \ref{fig:ind(iq)}. We also note that if $q$ is not an inflection point of $T$, then there are four distinct tangent lines to $T$ though $q$, hence $\iota_q$ has five indeterminacy points $q,q_1,q_2,q_3,q_4$ on $\bp^2$ (see Proposition 12 in \cite{Blanc2008}).

\begin{figure}[h]
\begin{tikzpicture}[xscale=0.5, yscale=0.5]
    \draw[thick,red!85!black,domain=-4.466:-1.6,samples=100,smooth,variable=\x] plot (\x,{-sqrt(5*(\x)^2+8*\x)});
    \draw[thick,red!85!black,domain=-4.466:-1.6,samples=100,smooth,variable=\x] plot (\x,{sqrt(5*(\x)^2+8*\x)});
    \draw[thick,red!85!black,domain=0:2.866,samples=100,smooth,variable=\x] plot (\x,{-sqrt(5*(\x)^2+8*\x)});
    \draw[thick,red!85!black,domain=0:2.866,samples=100,smooth,variable=\x] plot (\x,{sqrt(5*(\x)^2+8*\x)});
    \draw[thick,red!85!black] (-1.6,0.1) -- (-1.6,-0.1) node[left] at (-1.6,0) {$\fq$};
    \draw[very thick] (-1,0.1) -- (-1,-0.1);
    \draw[thick,blue!66!black] (-2.667,-8) -- (2.667,8) node[below] at (-2.667,-8) {$\fl_3$};
    \draw[thick,blue!66!black] (-2.667,8) -- (2.667,-8) node[above] at (-2.667,8) {$\fl_4$};
    \draw[thick,blue!66!black] (-5,5) -- (5,-5) node[right] at (5,-5) {$\fl_2$};
    \draw[thick,blue!66!black] (-5,-5) -- (5,5) node[right] at (5,5) {$\fl_1$};
    \draw[domain=0:2.64,samples=100,smooth,variable=\x,very thick] plot (\x,{sqrt(\x^3+5*\x^2+4*\x)});
    \draw[domain=0:2.64,samples=100,smooth,variable=\x,very thick] plot (\x,{-sqrt(\x^3+5*\x^2+4*\x)});
    \draw[domain=-4:-1,samples=100,smooth,variable=\x,very thick] plot (\x,{sqrt((\x)^3+5*(\x)^2+4*\x)});
    \draw[domain=-4:-1,samples=100,smooth,variable=\x,very thick] plot (\x,{-sqrt((\x)^3+5*(\x)^2+4*\x)});
    \draw[blue!66!black,fill] (2,6) circle (4pt) node[right] {$q_3$};
    \draw[blue!66!black,fill] (2,-6) circle (4pt) node[right] {$q_4$};
    \draw[blue!66!black,fill] (-2,2) circle (4pt) node[above right] {$q_2$};
    \draw[blue!66!black,fill] (-2,-2) circle (4pt) node[below right] {$q_1$};
    \draw[red!85!black,fill] (0,0) circle (3pt) node[right] {$q$};
\end{tikzpicture}
\caption{Curves contracted by $\iota_q$}
\label{fig:ind(iq)}
\end{figure}
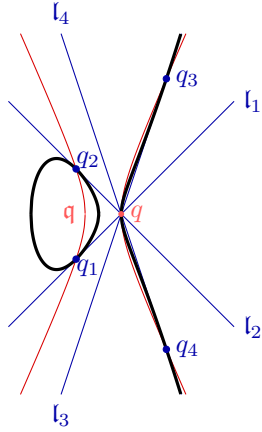

\begin{example}
\label{exp:6.1}
Let $T$ be the elliptic curve on $\bp^2$ over a field $K$ of characteristic $\Char{K}=0$ given by
\begin{equation*}
	y^2z=x(x+z)(x+4z)
\end{equation*}
and $q=[0:0:1]\in T$. The cubic involution associated with $T$ and $q$ is $\iota_q:\bp^2\dashrightarrow\bp^2$ defined by
\begin{equation}
\label{eqn:cubicinv}
	[x:y:z]\mapsto[5x^3-xy^2+8x^2z: 5x^2y-y^3+8xyz: -2x^3-5x^2z+y^2z].
\end{equation}
The involution $\iota_q$ contracts the four tangent lines
\begin{align*}
	\fl_1:y=x,\quad\fl_2:y=-x,\quad\fl_3:y=3x,\quad\fl_4:y=-3x
\end{align*}
through $q$ to the points
\begin{equation*}
	q_1=[-2:-2:1],\quad q_2=[-2:2:1],\quad q_3=[2:6:1],\quad q_4=[2:-6:1]
\end{equation*}
respectively. The conic through $q,q_1,\dots,q_4$ is
\begin{equation*}
	\fq: y^2=5x^2+8xz,
\end{equation*}
and $\iota_q$ contracts $\fq$ to $q$.
\end{example}

Now let $\iota_q$ be a fixed cubic involution, and let $\iota_t$ be another cubic involution. We will see in Proposition \ref{prop:actions} that the sequence $\BBrc{\deg\Par*{(\iota_q\circ\iota_t)^n}}$ grows quadratically for general $t\neq q$. So, by varying $t\in T$, we get a family of birational maps of $\bp^2$, whose general member has quadratic degree growth, whereas the special map $\iota_q\circ\iota_q$ is the identity on $\bp^2$. Moreover, we consider another family of birational self-maps whose general member is a composition of four cubic involutions, while the special member is again the identity. In this second family, the general member has dynamical degree $\lambda>1$.

\subsection{\texorpdfstring{Involution of $\bp^2\times T$}{}}

We define the families of maps as birational maps of a threefold. Let $X$ be the threefold $\bp^2\times T\subset\bp^2\times\bp^2$. Fix a point $q\in T$ that is not an inflection point, so that $\iota_q$ has five distinct indeterminacy points. We further endow $T$ with the structure of an elliptic with origin $o$, and assume that $q$ is a nontrivial $2$-torsion point. Define an involution $\varphi:X\dashrightarrow X$ by
\begin{equation*}
    (p,t)\mapsto(\iota_q(p),t).
\end{equation*}
The map $\varphi$ can also be viewed as a family of birational involutions of $\bp^2$, i.e., $\BBrc{\varphi_t:=\iota_q:t\in T}$. Moreover, we define an involution $\psi:X\dashrightarrow X$ by
\begin{equation*}
    (p,t)\mapsto(\iota_t(p),t)
\end{equation*}
and an involution $\chi:X\dashrightarrow X$ by
\begin{equation*}
    (p,t)\mapsto(\iota_{-t}(p),t),
\end{equation*}
where $-t$ is the inverse of $t$ with respect to the group law on $T$. So we have two other families of maps $\BBrc{\psi_t:=\iota_t:t\in T}$ and $\BBrc{\chi_t:=\iota_{-t}:t\in T}$.

For general $t$, $\iota_t$ has five indeterminacy points, and $\iota_{-t}$ has five others. All these ten points are distinct from the five points in $\Ind(\iota_q)$. Let us take the open subset $T^\circ$ of $T$ consisting of the point $q$ and those points $t\neq q$ such that $\#\Ind(\iota_t)=\#\Ind(\iota_{-t})=5$ and that the indeterminacy points of $\iota_t,\iota_{-t}$ and $\iota_q$ are all distinct. We replace $T$ with $T^\circ$ and keep the same notations $X=\bp^2\times T^\circ$, $\varphi=\varphi\vert_{\bp^2\times T^\circ}$, $\psi=\psi\vert_{\bp^2\times T^\circ}$, and $\chi=\chi\vert_{\bp^2\times T^\circ}$. We are now ready to state the indeterminacy locus of the maps $\varphi$, $\psi$ and $\chi$.

The indeterminacy locus of $\varphi$ contains five curves
\begin{equation*}
    B=\BBrc{q}\times T^\circ\quad\text{and}\quad B_i=\BBrc{q_i}\times T^\circ
\end{equation*}
for $i=1,\dots,4$. The indeterminacy locus of $\psi$ contains two irreducible curves, namely, the curve
\begin{equation*}
    C=\BBrc{(p,t):p=t}
\end{equation*}
and the curve $C'$. For $t\in T^\circ$, the restriction of the curve $C'$ to the fiber $\bp^2\times\BBrc{t}$ consists of the tangent points $t_1,t_2,t_3,t_4$ associated with $t$. The indeterminacy locus of $\chi$ contains the curve
\begin{equation*}
    D=\BBrc{(p,t):p=-t}
\end{equation*}
and the curve $D'$ whose restriction to $\bp^2\times\BBrc{t}$ are the four tangent points associated with $-t$. We note that
\begin{equation*}
	B\cap C\cap D=\BBrc{(q,q)}\quad\text{and}\quad B_i\cap C'\cap D'=\BBrc{(q_i,q)}
\end{equation*}
for $i=1,\dots,4$. These are the only intersection points of the indeterminate curves.

\setcounter{example}{5}
\begin{example}[continuted]
\label{exp:6.2}
Let $T$ be an elliptic curve on $\ba^2$ over a field $K$ of characteristic $\Char{K}=0$ given by
\begin{equation*}
	b^2=a(a-\alpha_1)(a-\alpha_2),
\end{equation*}
where $\alpha_1,\alpha_2$ are distinct nonzero elements of $K$, and $q=(0,0)\in T$. We take the open subset $T^\circ$ of $T$ as discussed above. Let $X=\ba^2\times T^\circ\subset\ba^2\times\ba^2$ with coordinates $((x,y),(a,b))$.

The curve $B$ is defined in $\ba^2\times T^\circ$ by
\begin{align*}
	x=y=0.
\end{align*}
The curve $C$ is defined by
\begin{equation}
\label{eqn:curveC}
	x-a=y-b=0.
\end{equation}
The curve $D$ is defined by
\begin{equation*}
	x-a=y+b=0.
\end{equation*}

	For $B_1,\dots,B_4$, $C'$, and $D'$, we need explicit formulas for the points of tangency associated with a given point on the elliptic curve, and these are called the half points in \cite{BZ2018}. For example, using the formulas in Example 2.2 and earlier equations in \cite{BZ2018} (with $\alpha_3=0$) the curves $B_1$ is defined in $\ba^2\times T^\circ$ by
\begin{equation*}
	x+\sqrt{-\alpha_1}\sqrt{-\alpha_2}=y-(-\sqrt{-\alpha_1}+\sqrt{-\alpha_2})x=0.
\end{equation*}
The curves $C'$ and $D'$ can be obtained as the two irreducible components of the following reducible curve defined in $\ba^2\times T^\circ$ by
\begin{align*}
	x^4-4ax^3+4(\alpha_1+\alpha_2)ax^2-2\alpha_1\alpha_2x^2-4\alpha_1\alpha_2ax+\alpha_1^2\alpha_2^2=y^2-x(x-\alpha_1)(x-\alpha_2)=0.
\end{align*}
In fact, by our choice of $T^\circ$, for a fixed $a$ with $(a,\pm b)\in T^\circ$ the above equations cut out exactly eight points in $\ba^2\times\BBrc{(a,b),(a,-b)}$. Specifically, four of the eight points lie in the fiber $\ba^2\times\BBrc{(a,b)}$, and the other four points lie in the fiber $\ba^2\times\BBrc{(a,-b)}$. Varying $a$ the first group of four points make up the curve $C'$ in $\ba^2\times T^\circ$, and the other group of four points make up the curve $D'$. Explicitly, the curve $C'$ is defined in $\ba^2\times T^\circ$ by the equations
\begin{equation}
\label{eqn:curveC'}
\begin{split}
	y^2-x(x-\alpha_1)(x-\alpha_2)&=0\\
	3x^2a-x^2(\alpha_1+\alpha_2)-2xa(\alpha_1+\alpha_2)+2x\alpha_1\alpha_2+a\alpha_1\alpha_2-y^2-2yb&=0\\
	xy^2+6xyb-y^2a+2xa(\alpha_1\alpha_2-\alpha_1^2-\alpha_2^2)+(x+a)(\alpha_1^2\alpha_2+\alpha_1\alpha_2^2)-2yb(\alpha_1+\alpha_2)-2\alpha_1^2\alpha_2^2&=0.
\end{split}
\end{equation}
With these defining equations, a computer calculation shows that $C'$ is irreducible for a general pair $(\alpha_1,\alpha_2)$. The defining equations for $D'$ are obtained in a similar way.

We notice that the equations (\ref{eqn:curveC}) define a curve $\overline{C}$ in $\ba^2\times T$, and  the equations (\ref{eqn:curveC'}) define a curve $\overline{C}'$ in $\ba^2\times T$. Since $\iota_t$ has five distinct indeterminacy points whenever $t$ is not an inflection point of $T$, we see that $\overline{C}$ and $\overline{C}'$ intersect only in the fibers $\ba^2\times\BBrc{t}$ with $t$ an inflection point of $T$. Since we have removed these points from $T$ to obtain $T^\circ$, $C$ and $C'$ are disjoint in $\ba^2\times T^\circ$. Similarly, $D$ and $D'$ are disjoint in $\ba^2\times T^\circ$.

All the irreducible curves discussed in this example can be readily checked to be smooth.
\hfill\par
\end{example}

Now we return to the discussion of families of birational maps. Using $\varphi,\psi$ and $\chi$, we can define various families of birational self-maps. We will first study the composition $\varphi\circ\psi:X\dashrightarrow X$, i.e., the family of maps $\BBrc{\iota_q\circ\iota_t:t\in T^\circ}$. We will then take the composition $\vartheta:=\varphi\circ\chi\circ\varphi\circ\psi$, which is the family of maps $\BBrc{\iota_q\circ\iota_{-t}\circ\iota_q\circ\iota_t:t\in T^\circ}$.

Let us fix some notations for the rest of the discussion. Given a threefold $X$ with a fibration $X\to T$, let $X_t$ denote the fiber of $X$ over $t\in T$. For a map $f:X\dashrightarrow X$ preserving $X\to T$, let $f_t$ denote the restriction $f\vert_{X_t}$.

\subsection{\texorpdfstring{Indeterminate curves of $\varphi\circ\psi$}{}}

Let us write $\cb:=B\cup B_1\cup\dotsm\cup B_4$, and $\cc:=C\cup C'$. By Proposition 2.1 in \cite{Blanc2013}, the cubic involution $\iota_q:\bp^2\dashrightarrow\bp^2$ can be lifted to an automorphism $\tilde{\iota}_q$ on the blow-up $S$ of $\bp^2$ at the five points $p,p_1,\dots,p_4$. The blow-up of $X=\bp^2\times T^\circ$ at $\cb$ is canonically isomorphic to $S\times T^\circ$, and the map $\tilde{\iota}_q\times\text{id}_{T^\circ}$ provides a regularization of $\varphi$ on $\Bl_\cb{X}$. Furthermore, $\iota_t$ is regularized and lifted to an automorphism $\tilde{\iota}_t$ for general $t\in T$ on the blow-up of $\bp^2$ at $t,t_1,\dots,t_4$. Letting the parameters in the defining equations of $\tilde{\iota}_q$ vary with $t\in T^\circ$ yields the defining equations of the automorphisms $\tilde{\iota}_t$. Together, these fiberwise automorphisms define an automorphism of $\Bl_\cc{X}$. Since $\varphi$ fixes the curves in $\cc$ pointwise, the lift of $\varphi$ on $\Bc\Bb{X}$ is also an automorphism. Similary, the lift of $\psi$ is an automorphism on $\Bb\Bc{X}$ (cf. Proposition 2.2 in \cite{Blanc2013} and Section 4.4.4 in \cite{Bot2021}).

Now we consider the lifts of $\varphi$ and $\psi$ on $\Bc\Bb{X}$, say $\tphi$ and $\tpsi$. The observation above suggests that $\tphi$ is an automorphism on $\Bc\Bb{X}$. The map $\tpsi$ is in fact a pseudo-automorphism on $\Bc\Bb{X}$. Indeed, although $X=\bp^2\times T^\circ$ is not complete, the map $\Bc\Bb{X}\dashrightarrow\Bb\Bc{X}$ is obtained explicitly by interchanging the order of blow-ups locally around each (transverse) intersection point of the curves $\cb$ and $\cc$. Near each such point the map is locally the Atiyah flop, and hence it factors as the composition of these Atiyah flops. The map $\tpsi$ is conjugate to the automorphism $\tpsi^+:\Bb\Bc{X}\to\Bb\Bc{X}$, the lift of $\psi$ on $\Bb\Bc{X}$, by these Atiyah flops with flopping curves in the fiber over $q$ as the intersection points are contained in the fiber over $q$. Therefore, $\tpsi$ has indeterminacy on the fiber over $q$. Consequently, the map $\tphi\circ\tpsi$ has no indeterminacy on fibers over $t\neq q$, while it has indeterminate curves on the fiber over $q$.

\subsubsection{\texorpdfstring{Blow-up, first round}{}}
\label{subsec:blow-upfirstround}

We denote by $\Bb{X}$ the blow-up of $X$ at $\cb$ with exceptional divisors $E_B$ and $E_{B_i}$ over $B$ and $B_i$ respectively. For $t\in T^\circ$, let $(\Bb{X})_t$ be the fiber of $\Bb{X}$ over $t$. Then $(\Bb{X})_t$ is the blow-up of $X_t\cong\bp^2$ at the five points
\begin{equation*}
	(q,t),(q_1,t),(q_2,t),(q_3,t),(q_4,t)\in X_t=\bp^2\times\BBrc{t},
\end{equation*}
where $q,q_1,q_2,q_3,q_4$ are the indeterminacy points of $\iota_t=\iota_q$, regarded as a birational involution of $\bp^2$. The involution $\varphi$ lifts to a regular involution $\hat{\varphi}$ on $\Bb{X}$. On each fiber $(\Bb{X})_t$, $\hat{\varphi}\vert_{(\Bb{X})_t}:(\Bb{X})_t\to(\Bb{X})_t$ swaps the exceptional curve over $(q,t)$ and the strict transform of the conic curve $\fq\times\BBrc{t}$, and exchanges the other four exceptional curves and the strict transforms of the four tangent lines respectively. In particular, on the fiber $(\Bb{X})_q$, we have the five exceptional curves
\begin{equation*}
    \fb:=E_B\cap(\Bb{X})_q\quad\text{and}\quad \fb_i:=E_{B_i}\cap(\Bb{X})_q
\end{equation*}
for $i=1,\dots,4$. Moreover, we still denote by $\mathfrak{q}$ the strict transform of the curve $\fq\times\BBrc{q}\subset X_q$ on $\Bb{X}$, and by $\fl_i$ the strict transforms of the curves $\fl_i\times\BBrc{q}\subset X_q$ on $\Bb{X}$ for $i=1,\dots,4$. We have
\begin{equation*}
    \hat{\varphi}(\fb)=\fq\quad\text{and}\quad\hat{\varphi}(\fb_i)=\fl_i
\end{equation*}
for $i=1,\dots,4$.

Let $\Bc{X}$ be the blow-up of $X$ at $\cc$ with the exceptional divisors $E_{C}^+$ and $E_{C'}^+$ over $C$ and $C'$  respectively. The map $\psi$ lifts to an automorphism $\hat{\psi}^+$ on $\Bc{X}$. On the fiber $(\Bc{X})_q$ over $q$, there are the exceptional curves
\begin{equation*}
    \fc^+=E_C^+\cap(\Bc{X})_q\quad\text{and}\quad\fc_1^+\sqcup\fc_2^+\sqcup\fc_3^+\sqcup\fc_4^+=E_{C'}^+\cap(\Bc{X})_q;
\end{equation*}
and there are the strict transform $\fq^+$ of the conic curve and the strict transforms $\fl_i^+$ of the tangent lines for $i=1,\dots,4$. We have
\begin{equation*}
    \hat{\psi}^+(\fc^+)=\fq^+\quad\text{and}\quad\hat{\psi}^+(\fc_i^+)=\fl_i^+
\end{equation*}
for $i=1,\dots,4$.

\subsubsection{Blow-up, second round}

We want to relate $\Bb{X}$ and $\Bc{X}$.

We can blow-up $\Bb{X}$ at the strict transform of $\cc$ to get $\Bc\Bb{X}$. We still denote by $E_B$ and $E_{B_i}$ the exceptional divisors over $B$ and $B_i$, respectively, on $\Bc\Bb{X}$; and we denote by $E_{C}$ and $E_{C'}$ the exceptional divisors over $C$ and $C'$, respectively, on $\Bc\Bb{X}$. We write $\widetilde{X}:=\Bc\Bb{X}$.

In the second round of blow-ups, we actually blow-up each fiber of $\Bb{X}$ further at five additional points. For $t\neq q$, the fiber of $\tX$ over $t$ is isomorphic to the blow-up of $X_t\cong\bp^2$ at the ten distinct points $(\cb\cup\cc)\cap X_t$. However, since the curves $\cb$ and $\cc$ intersect in the fiber $X_q$ at the five points
\begin{equation*}
	(q,q),(q_1,q),(q_2,q),(q_3,q),(q_4,q),
\end{equation*}
the fiber $\tX_q$ is the blow-up of $X_q\cong\bp^2$ at five points and then at five infinitely near points. More specifically, $(\Bb{X})_q$ is the blow-up of $X_q$ at $(q,q),(q_1,q),(q_2,q),(q_3,q),(q_4,q)$ with exceptional curves $\fb,\fb_1,\fb_2,\fb_3,\fb_4$. The curves $\fb,\fb_1,\fb_2,\fb_3,\fb_4$ intersect the strict transforms of $C,C'$ transversally at the points $(q',q),(q_1',q),(q_2',q),(q_3',q),(q_4',q)$ in $(\Bb{X})_q$ as shown in Figure \ref{fig:BXq}. Therefore, $\tX_q$ is the blow-up of $(\Bb{X})_q$ at such points.

\begin{figure}[h]
\begin{tikzpicture}[scale=0.72]
	\draw[very thick] (0,0) ellipse (5 and 4.2);
	\node at (0,3.5) {$(\Bb{X})_q$};
    \draw[red!65!white,very thick] (-1,-0.75) -- (1,0.75) node[left] at (-0.8,-0.75) {$\fb$};
    \draw[red!66!black,very thick] (1,1.25) -- (3,2.75) node[left] at (1.2,1.25) {$\fb_3$};
    \draw[red!66!black,very thick] (1,-2.75) -- (3,-1.25) node[left] at (1.2,-2.75) {$\fb_4$};
    \draw[red!66!black,very thick] (-3,1.25) -- (-1,2.75) node[left] at (-2.8,1.25) {$\fb_2$};
    \draw[red!66!black,very thick] (-3,-2.75) -- (-1,-1.25) node[left] at (-2.8,-2.75) {$\fb_1$};
    \draw[very thick, dashed] (-0.75,1.25) to[bend left] (0,0);\draw[very thick] (0,0) to[bend right] (0.75,-1.25) node[below] {$C$};
    \draw[very thick, dashed] (-2.75,-0.75) to[bend left] (-2,-2);\draw[very thick] (-2,-2) to[bend right] (-1.25,-3.25) node[above right] at (-1.25-0.4,-3.25) {$C'$};
    \draw[very thick, dashed] (-2.75,3.25) to[bend left] (-2,2);\draw[very thick] (-2,2) to[bend right] (-1.25,0.75) node[below] {$C'$};
    \draw[very thick, dashed] (1.25,-0.75) to[bend left] (2,-2);\draw[very thick] (2,-2) to[bend right] (2.75,-3.25) node[above right] at (2.75-0.4,-3.25) {$C'$};
    \draw[very thick, dashed] (1.25,3.25) to[bend left] (2,2);\draw[very thick] (2,2) to[bend right] (2.75,0.75) node[below] {$C'$};
    \draw[green!45!black,fill] (2,2) circle (3pt) node[left] at (2,2+0.2) {$q_3'$};
    \draw[green!45!black,fill] (2,-2) circle (3pt) node[left] at (2,-2+0.2) {$q_4'$};
    \draw[green!45!black,fill] (-2,2) circle (3pt) node[left] at (-2,2+0.2) {$q_2'$};
    \draw[green!45!black,fill] (-2,-2) circle (3pt) node[left]  at (-2,-2+0.2) {$q_1'$};
    \draw[blue!66!black,fill] (0,0) circle (3pt) node[left] at (0,0+0.2) {$q'$};
\end{tikzpicture}
\begin{tikzpicture}[scale=0.72]
	\draw[very thick] (0,0.5) ellipse (5 and 4.2);
	\node at (0,4) {$\tX_q$};
    \draw[red!65!white,very thick] (-1,-0.75+0.5) -- (1,0.75+0.5) node[left] at (-0.8,-0.75+0.5) {$\fb$};
    \draw[red!66!black,very thick] (1,1.25+0.5) -- (3,2.75+0.5) node[left] at (1.2,1.25+0.5) {$\fb_3$};
    \draw[red!66!black,very thick] (1,-2.75+0.5) -- (3,-1.25+0.5) node[left] at (1.2,-2.75+0.5) {$\fb_4$};
    \draw[red!66!black,very thick] (-3,1.25+0.5) -- (-1,2.75+0.5) node[left] at (-2.8,1.25+0.5) {$\fb_2$};
    \draw[red!66!black,very thick] (-3,-2.75+0.5) -- (-1,-1.25+0.5) node[left] at (-2.8,-2.75+0.5) {$\fb_1$};
    \draw[very thick, dashed] (-0.75,1.25-0.5) to[bend left] (0,0-0.5);\draw[very thick] (0,0-0.5) to[bend right] (0.75,-1.25-0.5);
    \draw[very thick, dashed] (-2.75,-0.75-0.5) to[bend left] (-2,-2-0.5);\draw[very thick] (-2,-2-0.5) to[bend right] (-1.25,-3.25-0.5);
    \draw[very thick, dashed] (-2.75,3.25-0.5) to[bend left] (-2,2-0.5);\draw[very thick] (-2,2-0.5) to[bend right] (-1.25,0.75-0.5);
    \draw[very thick, dashed] (1.25,-0.75-0.5) to[bend left] (2,-2-0.5);\draw[very thick] (2,-2-0.5) to[bend right] (2.75,-3.25-0.5);
    \draw[very thick, dashed] (1.25,3.25-0.5) to[bend left] (2,2-0.5);\draw[very thick] (2,2-0.5) to[bend right] (2.75,0.75-0.5);
    \draw[very thick, dashed] (-0.75,1.25+1.5) to[bend left] (0,0+1.5);\draw[very thick] (0,0+1.5) to[bend right] (0.75,-1.25+1.5);
    \draw[very thick, dashed] (-2.75,-0.75+1.5) to[bend left] (-2,-2+1.5);\draw[very thick] (-2,-2+1.5) to[bend right] (-1.25,-3.25+1.5);
    \draw[very thick, dashed] (-2.75,3.25+1.5) to[bend left] (-2,2+1.5);\draw[very thick] (-2,2+1.5) to[bend right] (-1.25,0.75+1.5);
    \draw[very thick, dashed] (1.25,-0.75+1.5) to[bend left] (2,-2+1.5);\draw[very thick] (2,-2+1.5) to[bend right] (2.75,-3.25+1.5);
    \draw[very thick, dashed] (1.25,3.25+1.5) to[bend left] (2,2+1.5);\draw[very thick] (2,2+1.5) to[bend right] (2.75,0.75+1.5);
    \draw[very thick, dashed] (-0.75,1.25-0.5) -- (-0.75,1.25+1.5) node[right] at (-0.75-0.2,1.25+0.3) {$E_C$};
    \draw[very thick, dashed] (-2.75,-0.75-0.5) -- (-2.75,-0.75+1.5);
    \draw[very thick, dashed] (-2.75,3.25-0.5) -- (-2.75,3.25+1.5) node[left] at (-2.75+0.2,3.25-0.5) {$E_{C'}$};
    \draw[very thick, dashed] (1.25,-0.75-0.5) -- (1.25,-0.75+1.5);
    \draw[very thick, dashed] (1.25,3.25-0.5) -- (1.25,3.25+1.5) node[left] at (1.25+0.2,3.25-0.5)  {$E_{C'}$};
    \draw[very thick, dashed] (0.75,-1.25-0.5) -- (0.75,-1.25+1.5);
    \draw[very thick, dashed] (-1.25,-3.25+1.5) -- (-1.25,-3.25-0.5) node[right] at (-1.25-0.2,-3.25+1.5) {$E_{C'}$};
    \draw[very thick, dashed] (-1.25,0.75+1.5) -- (-1.25,0.75-0.5);
    \draw[very thick, dashed] (2.75,-3.25+1.5) -- (2.75,-3.25-0.5) node[right] at (2.75-0.2,-3.25+1.5) {$E_{C'}$};
    \draw[very thick, dashed] (2.75,0.75+1.5) -- (2.75,0.75-0.5);
    \draw[blue!66!black,very thick] (0,-0.5) -- (0,1.5) node[left] at (0.5,-0.5) {$\fc$};
    \draw[green!45!black,very thick] (2,1.5) -- (2,3.5) node[left] at (2.15,1.5) {$\fc_3$};
    \draw[green!45!black,very thick] (2,-2.5) -- (2,-0.5) node[left] at (2.15,-2.5) {$\fc_4$};
    \draw[green!45!black,very thick] (-2,2-0.5) -- (-2,2+1.5) node[left] at (-1.85,1.5) {$\fc_2$};
    \draw[green!45!black,very thick] (-2,-2-0.5) -- (-2,-2+1.5) node[left] at (-1.85,-2.5) {$\fc_1$};
    \draw[red!66!black,fill] (2,2.5) circle (2pt);
    \draw[red!66!black,fill] (2,-1.5) circle (2pt);
    \draw[red!66!black,fill] (-2,2.5) circle (2pt);
    \draw[red!66!black,fill] (-2,-1.5) circle (2pt);
    \draw[red!65!white,fill] (0,0.5) circle (2pt);
\end{tikzpicture}
\caption{Blow-ups of $X_q$}
\label{fig:BXq}
\end{figure}
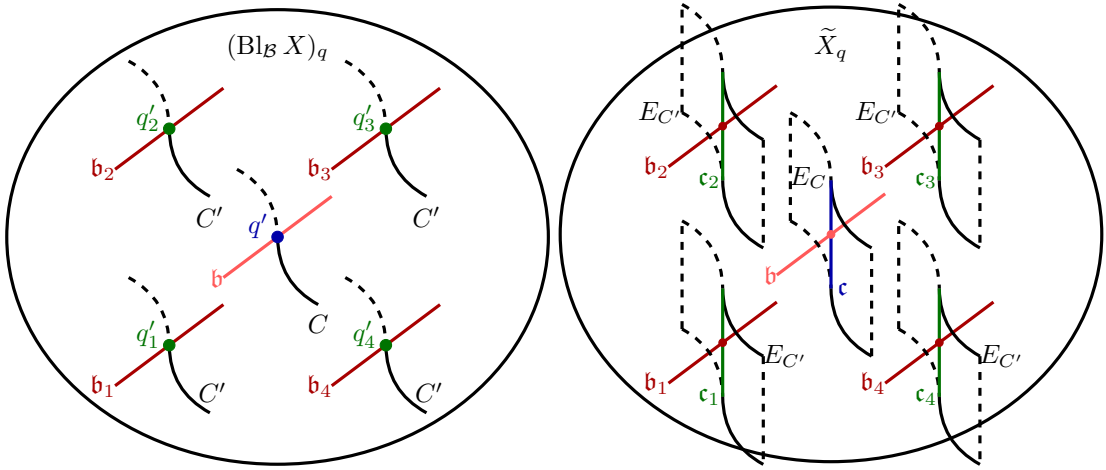

Let
\begin{equation*}
    \fc:=E_{C}\cap \tX_q\quad\text{and}\quad\fc_1\sqcup\fc_2\sqcup\fc_3\sqcup\fc_4:=E_{C'}\cap \tX_q.
\end{equation*}
We still write $\fb$ for the strict transform of $\fb$ in $\tX$, and $\fb_i$ for the strict transform of $\fb_i$ in $\tX$ for each $i$.
Since the map $\widehat{\varphi}$ fixes the strict transform $C$ and $C'$ pointwise, it lifts to an automorphism $\tphi$ on $\tX$, and $\tphi$ preserves the new exceptional divisors $E_{C},E_{C'}$ as $C$ and $C'$ are fixed by $\varphi$. On $\tX_q$, we have
\begin{equation*}
	\tphi(\fc)=\fc\quad\text{and}\quad\tphi(\fc_i)=\fc_i
\end{equation*}
for $i=1,\dots,4$, and
\begin{equation*}
	\tphi(\fb)=\fq\quad\text{and}\quad\tphi(\fb_i)=\fl_i,
\end{equation*}
where we use the same notations for the strict transforms of the conic curve and tangent lines on $\tX$.

On the other hand, we can blow-up $\Bc{X}$ at the strict transform of $\cb$ to get $\tX^+:=\Bb\Bc{X}$. The exceptional divisors over $B$ and $B_i$ in $\tX^+$ are $E_B^+$ and $E_{B_i}^+$ respectively, and the exceptional divisors over $C$ and $C'$ in $\tX^+$ are $E_{C}^+$ and $E_{C'}^+$ respectively. On the fiber $(\tX^+)_q$, we have
\begin{equation*}
    \fb^+:=E_B^+\cap\tX^+_q\quad\text{and}\quad\fb_i^+=E_{B_i}^+\cap\tX^+_q
\end{equation*}
for $i=1,\dots,4$; and we again write $\fc^+$ and $\fc_i^+$ for the strict transforms of $\fc^+$ and $\fc_i^+$ in $\tX^+$ for each $i$. The map $\hat{\psi}^+$ lifts to an automorphism $\tpsi^+$ on $\tX^+$, which preserves the new exceptional divisors $E_{B}^+,E_{B_1}^+,E_{B_2}^+,E_{B_3}^+,E_{B_4}^+$. In particular, we have
\begin{equation*}
	\tpsi^+(\fb^+)=\fb^+\quad\text{and}\quad\tpsi^+(\fb_i^+)=\fb_i^+,
\end{equation*}
for $i=1,\dots,4$, and
\begin{equation*}
	\tpsi^+(\fc^+)=\fq^+\quad\text{and}\quad\tpsi^+(\fc_i^+)=\fl_i^+,
\end{equation*}
for $i=1,\dots,4$.

\subsubsection{\texorpdfstring{Factorization of $\tilde{\varphi}\circ\tilde\psi$}{}}

Now the blow-ups $\tX$ and $\tX^+$ are related by a map $\sigma:\tX\dashrightarrow\tX^+$, which is a composition of five Atiyah flops. The five pairwise disjoint flopping curves for $\sigma$ are $\fb,\fb_1,\fb_2,\fb_3,\fb_4$; and the five pairwise disjoint flopped curves, which are the flopping curves for $\sigma^{-1}$, are $\fc^+,\fc^+_1,\fc^+_2,\fc^+_3,\fc^+_4$. 

Let us denote by $\tpsi$ the lift of $\psi$ on $\tX$. We factorize the map $\tphi\circ\tpsi$ in terms of flops and isomorphisms as follows:
\begin{equation*}
\begin{tikzcd}
	\tX \arrow[r,dashed,"\tpsi"] \arrow[d,dashed,"\sigma"'] & \tX \rar["\tphi"] & \tX\\
    \tX^+ \rar["\tpsi^+"] & \tX^+ \arrow[u,dashed,"\sigma^{-1}"']
\end{tikzcd}.
\end{equation*}
Although $\tpsi^+$ is an automorphism on $\tX^+$, $\tpsi$ is a pseudo-automorphism on $\tX$, which is conjugate to $\tpsi^+$ by $\sigma$. We note that $\sigma$ is undefined at $\fb,\fb_1,\fb_2,\fb_3,\fb_4$, and $\sigma^{-1}$ is undefined at $\fc^+,\fc^+_1,\fc^+_2,\fc^+_3,\fc^+_4$. So the composition $\tpsi=\sigma^{-1}\circ\tpsi^+\circ\sigma:\tX\dashrightarrow\tX$ is undefined along the ten curves
\begin{equation*}
	\fb,\fb_1,\fb_2,\fb_3,\fb_4,\fq,\fl_1,\fl_2,\fl_3,\fl_4.
\end{equation*}
Since $\tphi$ is a regular involution exchanging $\fb,\fb_1,\fb_2,\fb_3,\fb_4$ and $\fq,\fl_1,\fl_2,\fl_3,\fl_4$, the composition $\tilde{\varphi}\circ\tilde\psi:\tX\dashrightarrow\tX$ is also a pseudo-automorphism that is undefined at the same ten curves.

\subsection{Dynamics on general fibers}

For any divisor, we use the same notation to denote its class in the (relative) Néron-Severi space $N^1$.

Let $\pi_1:\bp^2\times T^\circ\to\bp^2$ be the first projection, and let $\pi:\tX\to X$ be the blow-up map. For a hyperplane divisor $h\in N^1(\bp^2)$, we write $H:=\pi^*(\pi_1^*(h))\in N^1(\tX/T^\circ)$. For any $t\neq q$, we write $h_t$ for $H\vert_{\tX_t}$. To distinguish from the exceptional curves in the fiber over $q$, we denote by $b_t:=E_B\vert_{\tX_t}$, $c_t:=E_C\vert_{\tX_t}$, $b_{i,t}:=E_{B_i}\vert_{\tX_t}$, $c_{i,t}:=E_{C_i}\vert_{\tX_t}$, $i=1,\dots,4$, the classes of the exceptional curves of $\tX_t$. Together with $h_t$, they form a set of basis elements for $N^1(\tX_t)$:
\begin{equation*}
	\BBrc{h_t,b_t,b_{1,t},b_{2,t},b_{3,t},b_{4,t},c_t,c_{1,t},c_{2,t},c_{3,t},c_{4,t}}.
\end{equation*}
We note that $\tphi_t=\tphi\vert_{\tX_t}$ and $\tpsi_t=\tpsi\vert_{\tX_t}$ are the lifts of $\iota_q$ and $\iota_t$, respectively, on the blow-up of $\bp^2$ at the ten indeterminacy points. Both maps are regular, and their induced action on $N^1(\tX_t)$ was computed in \cite{Bot2021}. The action matrices can also be obtained from the matrix in Lemma 3.3 of \cite{Blanc2013} by adding a diagonal block equal to the $5\times 5$ identity matrix, as $\iota_q$ fixes the points in $\Ind(\iota_t)$, and vice versa. In particular, we have
\begin{prop}
\label{prop:actions}
For $t\neq q$, the map $(\tphi_t)^*:N^1(\tX_t)\to N^1(\tX_t)$ is given by
\begin{equation*}
M_{\varphi}:=
\scriptsize{
\begin{bmatrix}
3 & 2 & 1 & 1 & 1 & 1 & 0 & 0 & 0 & 0 & 0 \\
-2 & -1 & -1 & -1 & -1 & -1 & 0 & 0 & 0 & 0 & 0 \\
-1 & -1 & -1 & 0 & 0 & 0 & 0 & 0 & 0 & 0 & 0 \\
-1 & -1 & 0 & -1 & 0 & 0 & 0 & 0 & 0 & 0 & 0 \\
-1 & -1 & 0 & 0 & -1 & 0 & 0 & 0 & 0 & 0 & 0 \\
-1 & -1 & 0 & 0 & 0 & -1 & 0 & 0 & 0 & 0 & 0 \\
0 & 0 & 0 & 0 & 0 & 0 & 1 & 0 & 0 & 0 & 0 \\
0 & 0 & 0 & 0 & 0 & 0 & 0 & 1 & 0 & 0 & 0 \\
0 & 0 & 0 & 0 & 0 & 0 & 0 & 0 & 1 & 0 & 0 \\
0 & 0 & 0 & 0 & 0 & 0 & 0 & 0 & 0 & 1 & 0 \\
0 & 0 & 0 & 0 & 0 & 0 & 0 & 0 & 0 & 0 & 1
\end{bmatrix}
},
\end{equation*}
and the map $(\tpsi_t)^*:N^1(\tX_t)\to N^1(\tX_t)$ is given by
\begin{equation*}
M_{\psi}:=
\scriptsize{
\begin{bmatrix}
3 & 0 & 0 & 0 & 0 & 0 & 2 & 1 & 1 & 1 & 1 \\
0 & 1 & 0 & 0 & 0 & 0 & 0 & 0 & 0 & 0 & 0 \\
0 & 0 & 1 & 0 & 0 & 0 & 0 & 0 & 0 & 0 & 0 \\
0 & 0 & 0 & 1 & 0 & 0 & 0 & 0 & 0 & 0 & 0 \\
0 & 0 & 0 & 0 & 1 & 0 & 0 & 0 & 0 & 0 & 0 \\
0 & 0 & 0 & 0 & 0 & 1 & 0 & 0 & 0 & 0 & 0 \\
-2 & 0 & 0 & 0 & 0 & 0 & -1 & -1 & -1 & -1 & -1 \\
-1 & 0 & 0 & 0 & 0 & 0 & -1 & -1 & 0 & 0 & 0 \\
-1 & 0 & 0 & 0 & 0 & 0 & -1 & 0 & -1 & 0 & 0 \\
-1 & 0 & 0 & 0 & 0 & 0 & -1 & 0 & 0 & -1 & 0 \\
-1 & 0 & 0 & 0 & 0 & 0 & -1 & 0 & 0 & 0 & -1
\end{bmatrix}
}.
\end{equation*}
\end{prop}
\begin{cor}
The eigenvalues of $M_{\psi}M_{\varphi}$ have absolute value $1$, and the largest Jordan block of $M$ is $\scriptsize{\begin{bmatrix}1&1&0\\0&1&1\\0&0&1\end{bmatrix}}$. Hence, the sequence $\BBrc{\deg((\tphi\circ\tpsi)_t^n)}_n$ grows quadratically for $t\neq q$.
\end{cor}
On the other hand, the map $(\tphi\circ\tpsi)_q$ is the identity map of the fiber $\tX_q$. Write $\fh:=H\vert_{\tX_q}$. By the formula (\ref{eqn:differenceformula}) in Proposition \ref{prop:formula}, the difference
\begin{equation*}
	(\tphi\circ\tpsi)^*(H)\vert_{\tX_q}-\Par{(\tphi\circ\tpsi)_q}^*(\fh)=2\fq+\sum\fl_i+4\fb+2\sum\fb_i
\end{equation*}
is a sum of the flopping curves (indeterminate curves); the indeterminate curves contribute the difference between
\begin{equation*}
	\deg_{\fh}((\tphi\circ\tpsi)_q^n)\quad\text{and}\quad\deg_{h_t}((\tphi\circ\tpsi)_t^n)
\end{equation*}
for $t\neq q$. In the rest of the paper, we will show that the dynamical complexity persists but is concentrated around the indeterminacy locus. Once we blow-up at the indeterminate curves and then look at the induced maps on the exceptional divisors, the interesting dynamics will be seen again.

As for the composition of four involutions $\vartheta=\varphi\circ\chi\circ\varphi\circ\psi$, i.e., the family $\BBrc{\iota_q\circ\iota_{-t}\circ\iota_q\circ\iota_t:t\in T^\circ}$, the map $\iota_q\circ\iota_{-t}\circ\iota_q\circ\iota_t$ can be regularizaed and has dynamical degree $17+12\sqrt{2}$ for $t\neq q$, while the limiting map is again the identity map when $t=q$. The fact that dynamical degree drops on the special fiber in a family of birational maps is consistent with the lower semicontinuity theorem on dynamical degree by \cite{Xie2015}. We will also blow-up at the indeterminate curves of $\vartheta$ in the fiber over $q$. The interesting dynamics will be recovered on the exceptional divisors too.

\section{Induced Map on Divisors}
\label{sec:4}

In this section we aim to blow-up $\tX$ at $\fq$ and obtain (a map which is conjugate to) the birational self-map $E_{\fq}\dashrightarrow E_{\fq}$ induced by $\tphi\circ\tpsi$.

\subsection{Some reductions}

To look at the exceptional divisor $E_\fq$, it suffices to just blow-up $X=\bp^2\times T^\circ$ at $\fq\subset X$, rather than blowing-up $\tX$ at the strict transform of $\fq$ on $\tX$. Furthermore, we just want to blow-up $\bp^2\times T$ at $\fq\subset\bp^2\times\BBrc{q}$. Indeed, $X=\bp^2\times T^\circ$ is an open subset of $\bp^2\times T$, and the blow-up center $\fq$ is entirely contained in $X$. Since blow-up commutes with restriction to an open subset, the blow-up of $X$ at $\fq$ is the restriction of the blow-up of $\bp^2\times T$ at $\fq$; and the exceptional divisor stays the same after restriction. Since we only need to consider what happens on the exceptional divisor $E_\fq$ and find the induced map on $E_\fq$, we perform our calculations on $\bp^2\times T$ for simplicity.

Next, we note that the induced map can be broken down into the composition of maps
\begin{equation*}
\begin{tikzcd}
	E_\fq \rar[dashed,"\bpsi^{-1}"] & E_{q} \rar[dashed,"\bphi"] & E_\fq
\end{tikzcd}
\end{equation*}
induced by $\varphi$ and $\psi$, where $E_{q}$ is the exceptional divisor of the blow-up of $X$ at the point $(q,q)\in X$. In fact, by Proposition \ref{prop:linkblowups}, there is a birational map from $E_q$ to $E_\fb$ (the exceptional divisor over the flopping curve $\fb$) induced by the identity map.

Finally, to see the dynamical property of the induced map $\bphi\circ\bpsi^{-1}$, we will analyze the map $\Phi:=\bpsi^{-1}\circ\bphi:E_{q}\dashrightarrow E_{q}$, which is birationally conjugate to $\bphi\circ\bpsi^{-1}:E_\fq\dashrightarrow E_\fq$ by $\bphi$. As $E_{q}$ is isomorphic to $\bp^2$ and $E_\fq$ is isomorphic to $\fq\times\bp^1$, it is simpler to study the map $\bpsi^{-1}\circ\bphi$ than the map $\bphi\circ\bpsi^{-1}$.

The map $\vartheta=\varphi\circ\chi\circ\varphi\circ\psi$ induces a birational self-map on $E_\fq\subset\Bl_\fq{X}$ as well, and it can be broken down into the composition
\begin{center}
\begin{tikzcd}
	E_\fq \arrow[r,dashed, "\bpsi^{-1}"] & E_q \arrow[r,dashed,"\bphi"] & E_\fq \arrow[r,dashed,"\bchi^{-1}"] & E_q \arrow[r,dashed,"\bphi"] & E_\fq.
\end{tikzcd}
\end{center}
Again, we study the map $\Psi:=\bchi^{-1}\circ\bphi$ which is birationally conjugate to $\bphi\circ\bchi^{-1}$ by $\bphi$. Therefore, the map $\Psi\circ\Phi$ is birationally conjugate to $\bphi\circ\bchi^{-1}\circ\bphi\circ\bpsi^{-1}$ by $\bphi$.

\subsection{Explicit formulas for cubic involution}

In order to obtain simpler formulas, we now fix the base field $K=\overline{\bq}$, a specific cubic curve $T$ given by
\begin{equation*}
    y^2z=x^3+5x^2z+4xz^2
\end{equation*}
in $\bp^2$ and the point $q=[0,0,1]\in T$. This is the elliptic curve in Example \ref{exp:6.1}. One of the affine pieces of $T$ is depicted in Figure \ref{fig:ind(iq)}.

Recall that the cubic involution $\iota_q:\bp^2\dashrightarrow\bp^2$ associated with $T$ and $q$ is defined by the formula (\ref{eqn:cubicinv}). The involution $\iota_q$ contracts the four tangent lines $\fl_1,\dots,\fl_4$ through $q$ to the points $q_1,\dots,q_4$ respectively. The conic through $q,q_1,\dots,q_4$ is
\begin{equation*}
	\fq: y^2=5x^2+8xz,
\end{equation*}
and $\iota_q$ contracts $\fq$ to $q$.

Now let $\bp^2\times\bp^2$ be the product of two projective planes with coordinates $([x:y:z],[a:b:c])$, and let $\bp^2\times T$ be the threefold defined in $\bp^2\times\bp^2$ by the equation
\begin{equation*}
	b^2c=a^3+5a^2c+4ac^2.
\end{equation*}
We define the involution $\varphi:\bp^2\times T\dashrightarrow\bp^2\times T$ by
\begin{equation*}
	\Par*{[x:y:z],[a:b:c]}\mapsto\Par*{[5x^3-xy^2+8x^2z: 5x^2y-y^3+8xyz: -2x^3-5x^2z+y^2z],[a:b:c]}.
\end{equation*}
We also write down the defining equations for $\psi:\bp^2\times T\dashrightarrow\bp^2\times T$ explicitly:
\begin{equation*}
	\Par*{[x:y:z],[a:b:c]}\mapsto\Par*{[P_1(x,y,z;a,b,c):P_2(x,y,z;a,b,c):P_3(x,y,z;a,b,c)],[a:b:c]},
\end{equation*}
where
\begin{align*}
	P_1&=(x^3 + 2y^2z - 4xz^2)a+ (-2xyz)b+ (5x^3 - xy^2 + 8x^2z)c,\\
	P_2&=(3x^2y + 10xyz + 4yz^2)a+ (- 2x^3 - 10x^2z - 8xz^2)b+ (5x^2y - y^3 + 8xyz)c,\\
	P_3&=(3x^2z + 10xz^2 + 4z^3)a+ (- 2yz^2)b+ (- 2x^3 - 5x^2z + y^2z)c.	
\end{align*}
For the elliptic curve we have chosen, the inverse of the point $[a:b:c]$ has coordinates $[a:-b:c]$. Therefore, the involution $\chi:\bp^2\times T\dashrightarrow\bp^2\times T$ is defined by
\begin{equation*}
	\Par*{[x:y:z],[a:b:c]}\mapsto\Par*{[P_1(x,y,z;a,-b,c):P_2(x,y,z;a,-b,c):P_3(x,y,z;a,-b,c)],[a:b:c]}.
\end{equation*}

We use the map $\varphi$ to explain how to find induced maps.

\subsection{\texorpdfstring{Procedure for finding induced maps}{}}
\label{subsec:process}

It suffices to make local calculations for finding birational maps. Let $\ba^2\times\ba^2$ be the affine space with coordinates $((x,y),(a,b))$. One of the affine pieces of the threefold $\bp^2\times T$ is defined by the equation
\begin{equation*}
	b^2=a^3+5a^2+4a
\end{equation*}
in $\ba^2\times\ba^2$. We use $Y$ to denote this affine piece. The curve $\fq$ is defined by $5x^2-y^2+8x=a=b=0$ in $Y$, and the point $(q,q)$ is $((0,0),(0,0))\in Y$.

\subsubsection{\texorpdfstring{Blow-up of $Y$ at conic}{}}

We explicitly define one affine chart of the blow-up of $Y$ at $\fq$, denoted $\Bl_\fq{Y}$, in $\ba^2\times\ba^2\times\ba^2$, with coordinates $((u,v),(c,d),(r,s))$, by the equations
\begin{equation*}
	dr-(c^2+5c+4)s=(5u^2-v^2+8u)r-d=(5u^2-v^2+8u)s-c=0.
\end{equation*}
The projection map $\pi:\Bl_\fq Y\to Y$ is given by
\begin{equation*}
	((u,v),(c,d),(r,s))\mapsto((u,v),(c,d)),
\end{equation*}
and inverse of the projection maps $\pi^{-1}:Y\dashrightarrow\Bl_\fq Y$ is given by
\begin{equation*}
	((x,y),(a,b))\mapsto((x,y),(a,b),(b/(5x^2-y^2+8x),a/(5x^2-y^2+8x))).
\end{equation*}
The exceptional divisor over $\fq$ is denoted by $E_\fq$. We note in particular that (one of the affine charts of) $E_\fq$ is the subvariety
\begin{equation*}
	\BBrc{(u,v,c,d,r,s):v^2=5u^2+8u,c=d=0,s=0}\subset(\ba^2)^3.
\end{equation*}
So there is an isomorphism
\begin{equation*}
	E_\fq\cong\fq\times\ba^1=\BBrc{((u,v),r):v^2=5u^2+8u}\subset\ba^2\times\ba^1.
\end{equation*}

\subsubsection{\texorpdfstring{Blow-up of $Y$ at point}{}}

Let $\Bl_q {Y}$ be the blow-up of $Y$ at $(q,q)$. We want to represent the exceptional divisor $E_q$ over $(q,q)$ using curves, which can be considered as elements in $\bp(N_{(q,q)/Y})$. We take parametrized curves through the point $(q,q)=((0,0),(0,0))\in Y$ in different tangent directions. For example, for fixed $(z,w)\in\ba^2$, consider the parametrized curve $\gamma_{(z,w)}:T\to Y$, where $T:=\BBrc{(\alpha,\beta)\in\ba^2:\alpha^3+5\alpha^2+4\alpha}$, defined by
\begin{equation*}
	(\alpha,\beta)\mapsto((\beta z,\beta w),(\alpha,\beta)).
\end{equation*}
Then $\gamma_{(z,w)}$ has tangent vector $((z,w),(0,1))$ at $((0,0),(0,0))$. In this way, the curve $\gamma_{(z,w)}$ corresponds to a point $p_{(z,w)}\in E_q$. Different curves $\gamma_{(z,w)}, \gamma_{(z',w')}$ have different tangent vectors at $(q,q)$ and therefore correspond to different points $p_{(z,w)},p_{(z',w')}$ on $E_q$. Taking all different points $(z,w)$ of $\ba^2$, we obtain a family of parametrized curves $\Gamma:\ba^2\times T\to Y$ given by
\begin{equation*}
	((z,w),(\alpha,\beta))\mapsto((\beta z,\beta w),(\alpha,\beta)),
\end{equation*}
which correspond to a dense open subset of $E_q$. We identify the $\ba^2$ with (a dense open subset of) $E_q$, and now $\Gamma$ has parameter domain $E_q\times T$.

\subsubsection{Induced map}

The composition
\begin{equation*}
	E_q\times T\overset{\Gamma}{\longrightarrow}Y\overset{\varphi}{\longrightarrow}Y\overset{\pi^{-1}}{\longrightarrow}\Bl_\fq Y
\end{equation*}
is a family (parametrized by $E_q$) of parameterizations (all parameterized by $T$) of curves, each of them passing through a point on the exceptional divisor $E_\fq$.

Regard $E_q$ as $\ba^2$ with coordinates $(z,w)$, and $E_\fq$ as $\fq\times\ba^1$ with coordinates $((u,v),r)$. The correspondence between a pair $(z,w)$ and a point on $E_\fq$ gives the induced map $\bphi:E_q\dashrightarrow E_\fq$, which is in fact defined by
\begin{equation}
\label{eqn:affinephi}
	(z,w)\mapsto\Par*{\Par*{\frac{8z^2}{-5z^2+w^2},\frac{8zw}{-5z^2+w^2}},\frac{25z^4-10w^2z^2+w^4}{72z^5-80w^2z^3+8w^4z}}.
\end{equation}
Here we point out that $\bphi$ is a map into the ruled surface $E_\fq$. Indeed, the first two components satisfy the equation of the conic $v^2=5u^2+8u$. The inverse $\bphi^{-1}:E_\fq\to E_q$ is given by
\begin{equation*}
	\Par*{(u,v),r}\mapsto\Par*{\frac{25u^5 - 10u^3v^2 + v^4u}{8u(9u^4 - 10u^2v^2 + v^4)r},\frac{25u^4v - 10u^2v^3 + v^5}{8u(9u^4 - 10u^2v^2 + v^4)r}}.
\end{equation*}

\subsection{Projectivized induced maps}

The process in Section \ref{subsec:process} yields three induced maps $\bphi,\bpsi,\bchi:E_q\dashrightarrow E_\fq$ with inverses $\bphi^{-1},\bpsi^{-1},\bchi^{-1}:E_\fq\dashrightarrow E_q$. To analyze the dynamical property of these maps, we projectivize the exceptional divisors $E_\fq$ and $E_q$ and homogenize the maps. Now we have
\begin{equation*}
	E_q\cong\bp^2\quad\text{and}\quad E_\fq\cong\fq\times\bp^1.
\end{equation*}
Homogenize the equations in (\ref{eqn:affinephi}); the projectivized map $\overline{\varphi}:E_q\dashrightarrow E_\fq$ is given by
\begin{equation*}
	[x:y:z]\mapsto\Par*{[8x^2:8xy:(-5x^2 + y^2)],[(5x^2 - y^2)^2z:8x(x-y)(x+y)(3x-y)(3x+y)]}.
\end{equation*}
Since we will use the maps $\overline{\psi}^{-1}$ and $\overline{\chi}^{-1}$, we present their definitions here:
\begin{itemize}
	\item $\overline{\psi}^{-1}:E_\fq\dashrightarrow E_q$ is defined by
\begin{equation*}
	([a:b:c],[s,t])\mapsto
	\begin{bmatrix}
		(25a^5-10a^3b^2+ab^4)t+(80a^4b-16a^2b^3)s:\\
		(25a^4b-10a^2b^3+b^5)t+(72a^5-8ab^4)s:\\
		8a(a-b)(a+b)(3a-b)(3a+b)s
	\end{bmatrix}^T.
\end{equation*}
	\item $\overline{\chi}^{-1}:E_\fq\dashrightarrow E_q$ is defined by
\begin{equation*}
	([a:b:c],[s,t])\mapsto
	\begin{bmatrix}
		(25a^5-10a^3b^2+ab^4)t-(80a^4b-16a^2b^3)s:\\
		(25a^4b-10a^2b^3+b^5)t-(72a^5-8ab^4)s:\\
		8a(a-b)(a+b)(3a-b)(3a+b)s
	\end{bmatrix}^T.
\end{equation*}
\end{itemize}
The remaining maps can be found in a separate SageMath script.

\section{Dynamics of Induced Maps}
\label{sec:5}

In the following discussion, if no confusion arises, we always use the same notation for a point and the corresponding point on any blow-up; and we always use the same notation for a subvariety and all of its strict transforms. Moreover, by the observation (\ref{eqn:degreeff-1}), we use the pushforward $f_*$, rather than $f^*$, to analyze the degrees and dynamical degrees in what follows.

\subsection{\texorpdfstring{The map induced by $\varphi\circ\psi$}{}}

Recall that the map conjugate to the map induced by $\varphi\circ\psi$ is $\Phi:=\bpsi^{-1}\circ\bphi:E_q\cong\bp^2\dashrightarrow\bp^2\cong E_q$, and $\Phi$ is defined by
\begin{align*}
	[x:y:z]\mapsto
	\begin{bmatrix}
		9x^5 - 10x^3y^2 + xy^4 + 10x^3yz - 2xy^3z:\\
		9x^4y - 10x^2y^3 + y^5 + 9x^4z - y^4z:\\
		9x^4z - 10x^2y^2z + y^4z
	\end{bmatrix}^T.
\end{align*}
The inverse map $\Phi^{-1}:=\bphi^{-1}\circ\bpsi:E_q\dashrightarrow E_q$ is given by
\begin{equation*}
	[x:y:z]\mapsto
	\begin{bmatrix}
		x(9x^4 - 10x^2y^2 + y^4 + 10x^2yz - 2y^3z + 2yz^3 - z^4):\\
		(y - z)(9x^4 - 10x^2y^2 + y^4 + 10x^2yz - 2y^3z + 2yz^3 - z^4):\\
		z(x - y + z)(x + y - z)(3x - y + z)(3x + y - z)
	\end{bmatrix}^T.
\end{equation*}
We will show in Proposition \ref{prop:actionofPhi} that the sequence $\BBrc{\deg(\Phi^n)}$ grows quadratically.

A direct computation shows that the indeterminacy points of $\Phi$ are
\begin{equation*}
	p:=[0:0:1], p_1:=[1:1:0], p_2:=[1:-1:0], p_3:=[1:3:0], p_3:=[1:-3:0].
\end{equation*}
The indeterminacy points of $\Phi^{-1}$ are
\begin{equation*}
	p':=[0:1:1], p_1:=[1:1:0], p_2:=[1:-1:0], p_3:=[1:3:0], p_4:=[1:-3:0].
\end{equation*}
We write $P:=\BBrc{p,p_1,p_2,p_3,p_4,p'}$.

Let $\Bl_P E_q$ be the blow-up of $E_q$ at $P$ with exceptional divisors $e_p$, $e_{p_1}$, $e_{p_2}$, $e_{p_3}$, $e_{p_4}$, $e_{p'}$. Both the lift of $\Phi$ and the lift of $\Phi^{-1}$ on $\Bl_P E_q$ will still have four indeterminacy points $p_1',p_2',p_3',p_4'$, one on each of the exceptional divisors $e_{p_1},e_{p_2},e_{p_3},e_{p_4}$. We write $P':=\BBrc{p_1',p_2',p_3',p_4'}$. In fact, the pullback of a general line in $E_q\cong\bp^2$ under $\Phi$ passes through the points $p_1,p_2,p_3,p_4$ with respective tangent lines
\begin{align*}
	l_1:2x-2y+z=0,\quad l_2:2x+2y-z=0,\quad l_3:6x-2y+z=0,\quad l_4:6x+2y-z=0.
\end{align*}
The same claim holds for $\Phi^{-1}$.

The map $\Phi$ preserves a fibration on $E_q\cong\bp^2$ of singular quartic curves with two nodal points at $p,p'$ and with tangent lines $l_1,\dots,l_4$ at $p_1,\dots,p_4$, respectively. Now we blow-up $\Bl_P E_q$ at $P'$ to get $\Bl_{P'}\Bl_P E_q$ with exceptional divisors $e_{p_1'},e_{p_2'},e_{p_3'},e_{p_4'}$ at the corresponding points. Both $\Phi$ and $\Phi^{-1}$ lift to an automorphism on $\Bl_{P'}\Bl_P E_q$.

\subsection{\texorpdfstring{Quadratic degree growth of $\Phi$}{}}

Given the sequence of blow-ups
\begin{equation*}
	\Bl_{P'}\Bl_P E_q\to\Bl_P E_q\to E_q,
\end{equation*}
we write $\hat{\Phi}$ for the lift of $\Phi$ on $\Bl_P E_q$, and write $\tPhi$ for the lift of $\Phi$ on $\Bl_{P'}\Bl_P E_q$. To see the degree growth of $\Phi$, we examine the action of $\tPhi$ on $N^1\Par*{\Bl_{P'}\Bl_P E_q}$. Let $h$ be a hyperplane class. A set of (ordered) basis elements for $N^1(\Bl_{P'}\Bl_P E_q)$ is
\begin{equation*}
	\BBrc*{h,e_p,e_{p_1},e_{p_2},e_{p_3},e_{p_4},e_{p'},e_{p_1'},e_{p_2'},e_{p_3'},e_{p_4'}}.
\end{equation*}

\begin{prop}
\label{prop:actionofPhi}
The map $\tPhi_*:N^1\Par*{\Bl_{P'}\Bl_P E_q}\to N^1\Par*{\Bl_{P'}\Bl_P E_q}$ is given by
\begin{equation}
\label{eqn:actionmatrix}
M:=
\scriptsize{
\begin{bmatrix}
5 & 4 & 0 & 0 & 0 & 0 & 0 & 1 & 1 & 1 & 1 \\
0 & 0 & 0 & 0 & 0 & 0 & 1 & 0 & 0 & 0 & 0 \\
-1 & -1 & 1 & 0 & 0 & 0 & 0 & -1 & 0 & 0 & 0 \\
-1 & -1 & 0 & 1 & 0 & 0 & 0 & 0 & -1 & 0 & 0 \\
-1 & -1 & 0 & 0 & 1 & 0 & 0 & 0 & 0 & -1 & 0 \\
-1 & -1 & 0 & 0 & 0 & 1 & 0 & 0 & 0 & 0 & -1 \\
-4 & -3 & 0 & 0 & 0 & 0 & 0 & -1 & -1 & -1 & -1 \\
-2 & -2 & 0 & 0 & 0 & 0 & 0 & -1 & 0 & 0 & 0 \\
-2 & -2 & 0 & 0 & 0 & 0 & 0 & 0 & -1 & 0 & 0 \\
-2 & -2 & 0 & 0 & 0 & 0 & 0 & 0 & 0 & -1 & 0 \\
-2 & -2 & 0 & 0 & 0 & 0 & 0 & 0 & 0 & 0 & -1
\end{bmatrix}
}.
\end{equation}
\end{prop}

\begin{proof}
Let $h$ be a general hyperplane on $\bp^2$ define by $ax+by+cz=0$. The image of $h$ under $\Phi$ is defined by $a\Phi^{-1}_1(x,y,z)+b\Phi^{-1}_2(x,y,z)+c\Phi^{-1}_3(x,y,z)=0$, where $\Phi^{-1}_i$'s are the defining polynomials of $\Phi^{-1}$. This is a curve of degree $5$ that passes through the point $p'=[0:1:1]$ with multiplicity $4$ and the points $p_1,p_2,p_3,p_4$ with tangent lines $l_1,l_2,l_3,l_4$, respectively. On $\Bl_P E_q$, we have
\begin{equation*}
	\hat{\Phi}_*(h)=5h-4e_{p'}-\sum e_{p_i}.
\end{equation*}
On $\Bl_{P'}\Bl_P E_q$, we have
\begin{align*}
	\tPhi_*(h)&=5h-4e_{p'}-\sum(e_{p_i}+e_{p_i'})-\sum e_{p_i'}\\
	&=5h-4e_{p'}-\sum e_{p_i}-2\sum e_{p_i'}
\end{align*}

By computing the Jacobian of the inverse map $\Phi^{-1}$, we find that $\Phi^{-1}$ contracts a degree-$4$ curve to $p=[0:0:1]$. Moreover, this curve passes through the point $p'=[0:1:1]$ with multiplicity $3$ and through $p_1,p_2,p_3,p_4$ with tangent lines $l_1,l_2,l_3,l_4$, respectively. We have
\begin{equation*}
	\hat{\Phi}_*(e_p)=4h-3e_{p'}-\sum e_{p_i},
\end{equation*}
and so
\begin{align*}
	\tPhi_*(e_p)&=4h-3e_{p'}-\sum(e_{p_i}+e_{p_i'})-\sum e_{p_i'}\\
	&=4h-3e_{p'}-\sum e_{p_i}-2\sum e_{p_i'}.
\end{align*}

We note that $\Phi(0:1:1)=[0:0:1]$, and it can be verified that
\begin{equation*}
	\tPhi_*(e_{p'})=e_p.
\end{equation*}

Next, we observe that the image of a general line passing through $p_1=[1:1:0]$ is a degree-$4$ curve of passing through the point $p'=[0:1:1]$ with multiplicity $3$. The curve also passes through the point $p_1$ with tangent line \underline{not} equal to $l_1:2x-2y+z=0$ and through the points $p_2,p_3,p_4$ with tangent lines $l_2,l_3,l_4$, respectively. Thus,
\begin{equation*}
	\hat\Phi_*(h-e_{p_1})=4h-3e_{p'}-\sum e_{p_i},
\end{equation*}
and hence
\begin{equation}
\label{eqn:H-p1-p1'}
	\tPhi_*(h-e_{p_1}-e_{p_1'})
	=4h-3e_{p'}-\sum e_{p_i}-2\sum e_{p_i'}+e_{p_1'}.
\end{equation}

The image of the line $l_1:2x-2y+z=0$ is a curve of degree $3$ that passes through $p'=[0:1:1]$ with multiplicity $2$ and $p_2,p_3,p_4$ with tangent lines $l_2,l_3,l_4$, respectively. So
\begin{align}
\label{eqn:H-p1-2p1'}
\begin{split}
	\tPhi_*(h-e_{p_1}-2e_{p_1'})=3h-2e_{p'}-\sum e_{p_i}-2\sum e_{p_i'}+e_{p_1}+2e_{p_1'}.
\end{split}
\end{align}
Combining (\ref{eqn:H-p1-p1'}) and (\ref{eqn:H-p1-2p1'}) we get
\begin{equation}
\label{eqn:f(p1')}
	\tPhi_*(e_{p_1'})=h-e_{p'}-e_{p_1}-e_{p_1'}.
\end{equation}
Similarly,
\begin{equation*}
	\tPhi_*(e_{p_i'})=h-e_{p'}-e_{p_i}-e_{p_i'},
\end{equation*}
for $i=2,3,4$.

Finally, substituting (\ref{eqn:f(p1')}) back into (\ref{eqn:H-p1-p1'}) and rearranging terms, we get
\begin{align*}
	\tPhi_*(e_{p_1})&=\tPhi_*(h)-\tPhi_*(e_p)-\tPhi_*(e_{p_1'})-e_{p_1'}\\
	&=h-e_{p'}-e_{p_1'}-(h-e_{p'}-e_{p_1}-e_{p_1'})\\
	&=e_{p_1}.
\end{align*}
Similarly, we calculate that
\begin{equation*}
	\tPhi_*(e_{p_i})=e_{p_i},
\end{equation*}
for $i=2,3,4$.
\end{proof}

\begin{cor}
The eigenvalues of $M$ have absolute value $1$, and the largest Jordan block of $M$ is $\scriptsize{\begin{bmatrix}1&1&0\\0&1&1\\0&0&1\end{bmatrix}}$. Hence the sequence $\BBrc{\deg(\Phi^n)}_n$ grows quadratically.
\end{cor}

Thus the conjugate map $E_\fq\dashrightarrow E_\fq$ induced by $\varphi\circ\psi$ on the exceptional divisor over the curve $\fq$ in the fiber over $q$ has quadratic degree growth. Recall that the map $(\varphi\circ\psi)_t$ on a general fiber has quadratic degree growth, whereas the map $(\varphi\circ\psi)_q$ is the identity map. In this way, we obtain interesting dynamics on the exceptional divisor $E_\fq$.

Recall that the map $\vartheta_t$ on a general fiber has dynamical degree $17+12\sqrt{2}$. We expect the induced self-map of $E_\fq$ by $\vartheta$ to have $\lambda>1$ as well.

\subsection{\texorpdfstring{The map induced by $\vartheta$}{}}

Using the maps $\bchi$ and $\bphi$ (and their inverses), we obtain another map $\Psi:=\bchi^{-1}\circ\bphi:E_q\dashrightarrow E_q$ with the inverse map $\Psi^{-1}:=\bphi^{-1}\circ\bchi:E_q\dashrightarrow E_q$. The map $\Psi\circ\Phi$ is conjugate to the map induced by $\vartheta=\varphi\circ\chi\circ\varphi\circ\psi$. We will show that $\lambda(\Psi\circ\Phi)=16$, but first we need to understand $\Psi$.

The map $\Psi$ behaves almost the same as $\Phi$. The indeterminacy points of $\Psi$ are
\begin{equation*}
	p=[0:0:1], p_1=[1:1:0], p_2=[1:-1:0], p_3=[1:3:0], p_4=[1:-3:0].
\end{equation*}
Moreover, the indeterminacy points of $\Psi^{-1}$ are
\begin{equation*}
	p'':=[0:-1:1], p_1=[1:1:0], p_2=[1:-1:0], p_3=[1:3:0], p_4=[1:-3:0].
\end{equation*}
The pullback of a general line on $E_q$ under both $\Psi$ passes through the points $p_1,p_2,p_3,p_4$ with respective tangent lines
\begin{align*}
	2x-2y-z=0,\quad 2x+2y+z=0,\quad 6x-2y-z=0,\quad 6x+2y+z=0.
\end{align*}
The same claim holds for $\Psi^{-1}$. Therefore, if we blow-up $E_q$ at $Q:=\BBrc{p,p'',p_1,p_2,p_3,p_4}$, then both the lifts $\hat{\Psi}$ and $\hat{\Psi}^{-1}$ on $\Bl_QE_q$ will have an indeterminacy point $p_i''$ at the exceptional divisors $e_{p_i}$ for each $i=1,\dots,4$. If we blow-up again at $Q':=\BBrc{p_1'',p_2'',p_3'',p_4''}$, then $\hat{\Psi}$ and $\hat{\Psi}^{-1}$ will be regularized on $\Bl_{Q'}\Bl_QE_q$ and lift to $\tPsi$ and $\tPsi^{-1}$. The linear map $\tPsi_*:N^1(\Bl_{Q'}\Bl_QE_q)\to N^1(\Bl_{Q'}\Bl_QE_q)$ can be represented by the same matrix (\ref{eqn:actionmatrix}) with respect to a different basis, of course.

Now we are ready to study the dynamical degree of $\Psi\circ\Phi$. Let us blow-up $E_q$ at
\begin{equation*}
	R:=\BBrc{p,p',p'',p_1,p_2,p_3,p_4}
\end{equation*}
to get $\Bl_R\bp^2$ and then at
\begin{equation*}
	R':=\BBrc{p_1',p_2',p_3',p_4',p_1'',p_2'',p_3'',p_4''}
\end{equation*}
to get $\Bl_{R'}\Bl_RE_q$. Then $N^1\Par*{\Bl_{R'}\Bl_RE_q}$ has rank $16$ with an ordered basis
\begin{equation*}
	\BBrc*{h,e_p,e_{p'},e_{p''},e_{p_1},e_{p_2},e_{p_3},e_{p_4},e_{p_1'},e_{p_2'},e_{p_3'},e_{p_4'},e_{p_1''},e_{p_2''},e_{p_3''},e_{p_4''}}.
\end{equation*}
Let $\wtPhi$ and $\wtPsi$ denote the lifts of $\tPhi$ and $\tPsi$ on $\Bl_{R'}\Bl_RE_q$. The maps $\tPhi$ (resp. $\tPsi$) is an automorphism (which contracts no curves) on $\Bl_{P'}\Bl_PE_q$ (resp. $\Bl_{Q'}\Bl_QE_q$). Nevertheless, some curves on $\Bl_{R'}\Bl_RE_q$ will be contracted by $\wtPhi$ and $\wtPsi$. The map $\wtPhi$ contracts the curves $e_{p''},e_{p_1''},e_{p_2''},e_{p_3''},e_{p_4''}$, and the map $\wtPsi$ contracts the curves $e_{p'},e_{p_1'},e_{p_2'},e_{p_3'},e_{p_4'}$.

\begin{prop}
\label{prop:Mgf=MgMf}
For any curve $C$ contracted by $\wtPhi$, we have $\wtPhi(C)\not\subset\Ind(\wtPsi)$. We conclude that
\begin{equation*}
	(\wtPsi\circ\wtPhi)_*=\wtPsi_*\wtPhi_*.
\end{equation*}
\end{prop}
\begin{proof}
Since $\Phi(p'')=\Phi(0:-1:1)=[0:-2:1]=:p_{-2}$, $\widetilde{\Phi}(e_{p''})=p_{-2}$ on $\Bl_{R'}\Bl_RE_q$. Moreover, $\wtPsi$ is defined at $p_{-2}$. In fact, the point $p_{-2}$ lies on a curve which is mapped isomorphically to $e_{p''}$ by $\wtPsi$.

What happens at $e_{p_i''}$, for $i=1,\dots,4$, is similar, and we present only the local computations around $p_1$. The blow-up $\Bl_{p_1}E_q$ of $E_q$ at $p_1=[1:1:0]$ can be defined in $E_q\times\bp^1\cong\bp^2\times\bp^1$ (with coordinates $([x:y:z],[s:t])$) by the equation
\begin{equation*}
	(x-y)s=zt
\end{equation*}
with the exceptional divisor
\begin{equation*}
	e_{p_1}=\BBrc*{([1:1:0],[s:t])}\subset\Bl_{p_1}E_q.
\end{equation*}
In terms of the coordinates on $e_{p_1}$, $p''_1=([1:1:0],[2:1])$. Let $\Phi_1$ and $\Psi_1$ be the lifts of $\Phi$ and $\Psi$ on $\Bl_{p_1}E_q$. The restriction $\Phi_1\vert_{e_{p_1}}:e_{p_1}\to e_{p_1}$ is given by
\begin{equation*}
	([1:1:0],[s:t])\mapsto([1:1:0],[2s:2t-s]).
\end{equation*}
So $\Phi_1([1:1:0],[2:1])=([1:1:0],[1:0])$. This implies that $\wtPhi$ contracts $e_{p_1''}$ to the point corresponding to $([1:1:0],[1:0])$. Moreover, $\wtPsi$ is defined at the point corresponding to $([1:1:0],[1:0])$, and in fact, the point corresponding to $([1:1:0],[1:0])$ lies on a curve which is mapped isomorphically to $e_{p_1''}$ by $\wtPsi$.

By Theorem \ref{thm:AS}, we have $(\wtPsi\circ\wtPhi)_*=\wtPsi_*\wtPhi_*$.
\end{proof}

Consequently, $\wtPsi\circ\wtPhi$ contracts the curves
\begin{equation*}
	e_{p''},e_{p_1''},e_{p_2''},e_{p_3''},e_{p_4''},(\wtPhi^{-1})(e_{p'}),(\wtPhi^{-1})(e_{p_1'}),(\wtPhi^{-1})(e_{p_2'}),(\wtPhi^{-1})(e_{p_3'}),(\wtPhi^{-1})(e_{p_4'}).
\end{equation*}
on $\Bl_{R'}\Bl_RE_q$, where $(\wtPhi^{-1})(C)$ means the proper transform of a curve $C$ under the inverse of $\wtPhi$.

\subsection{\texorpdfstring{Algebraic stability of $\wtPsi\circ\wtPhi$}{}}

We prove the algebraic stability of $\wtPsi\circ\wtPhi$ on $\Bl_{R'}\Bl_RE_q$.

\begin{prop}
The map $\wtPsi\circ\wtPhi$ is algebraically stable on $\Bl_{R'}\Bl_RE_q$.
\end{prop}
\begin{proof}

By Theorem \ref{thm:AS}, it suffices to show that for any curve $C$ contracted by $\wtPsi\circ\wtPhi$, $(\wtPsi\circ\wtPhi)(C)$ is a fixed point of $\wtPsi\circ\wtPhi$.

We first deal with the curves $e_{p''},e_{p_1''},e_{p_2''},e_{p_3''},e_{p_4''}$.

For $e_{p''}$, we note that $\wtPhi$ is defined on all of $e_{p''}$ and $\wtPhi$ contracts $e_{p''}$ to the point $p_{-2}$ (corresponding to $[0:-2:1]$). Moreover, $\wtPsi$ is defined at $p_{-2}$, and it lies on a curve which is mapped isomorphically to $e_{p''}$ by $\wtPsi$. Therefore, $\wtPsi\circ\wtPhi$ is defined everywhere on $e_{p''}$ and maps $e_{p''}$ to a point on itself. In particular, $\wtPsi\circ\wtPhi$ is defined at $(\wtPsi\circ\wtPhi)(e_{p''})$ and $\wtPsi\circ\wtPhi$ fixes this point.

To see the points $p_i''$, we need to blow-up $\bp^2$ at $p_i$ for $i=1,\dots,4$. Again, we show the local calculations around $p_1$. Recall that in the proof of Proposition \ref{prop:Mgf=MgMf}, we have $p_1''=([1:1:0],[2:1])$ in terms of the coordinates on $e_{p_1}$ and the restriction $\Phi_1\vert_{e_{p_1}}:e_{p_1}\to e_{p_1}$ is given by
\begin{equation*}
	([1:1:0],[s:t])\mapsto([1:1:0],[2s:2t-s]).
\end{equation*}
Let $s_1$ be the point corresponding to $([1:1:0],[1:0])$. Then $\wtPhi$ is defined everywhere on $e_{p_1''}$ and contracts $e_{p_1''}$ to $s_1$. The map $\wtPsi$ is defined at $s_1$, and $s_1$ lies on a curve which is mapped isomorphically to $e_{p_1''}$ by $\wtPsi$, so $\wtPsi(s_1)$ is a point on $e_{p_1''}$. Therefore, $\wtPsi\circ\wtPhi$ is defined at $(\wtPsi\circ\wtPhi)(e_{p_1''})$ and fixes this point.

We next consider the curves $(\wtPhi^{-1})(e_{p'}),(\wtPhi^{-1})(e_{p_1'}),(\wtPhi^{-1})(e_{p_2'}),(\wtPhi^{-1})(e_{p_3'}),(\wtPhi^{-1})(e_{p_4'})$.

For $(\wtPhi^{-1})(e_{p'})$, we note that $\wtPhi$ is defined everywhere on $(\wtPhi^{-1})(e_{p'})$ and maps it isomorphically to $e_{p'}$. Since $\Psi(p')=\Psi([0:1:1])=[0:2:1]$, the map $\wtPsi$ is defined everywhere on $e_{p'}$ and contracts $e_{p'}$ to $p_{+2}$, the point corresponding to $[0:2:1]$ on $\Bl_{R'}\Bl_RE_q$. Moreover, the point $p_{+2}$ lies on $(\wtPhi^{-1})(e_{p'})$. Therefore, $\wtPsi\circ\wtPhi$ contracts $(\wtPhi^{-1})(e_{p'})$ and fixes its image point $p_{+2}$.

For each $i$, the the map $\wtPhi$ is defined everywhere on $(\wtPhi^{-1})(e_{p_i'})$ and maps it isomorphically to $e_{p_i'}$. It suffices to look at $e_{p_1'}$. Similarly as in the proof of Proposition \ref{prop:Mgf=MgMf} we have $p_1'=([1:1:0],[-2:1])$, and the restriction $\Psi_1\vert_{e_{p_1}}$ is given by
\begin{equation*}
	([1:1:0],[s:t])\mapsto([1:1:0],[2s:2t+s]).
\end{equation*}
So the map $\wtPsi$ is defined everywhere on $e_{p_1'}$ and contracts $e_{p_1'}$ to $s_1$. Moreover, the point $s_1$ lies on $(\wtPhi^{-1})(e_{p_1'})$. Therefore $\wtPsi\circ\wtPhi$ contracts $(\wtPhi^{-1})(e_{p_1'})$ and fixes its image point $s_1$.
\end{proof}

\subsection{\texorpdfstring{Dynamical degree of $\Psi\circ\Phi$}{}}

The blow-up $\Bl_{R'}\Bl_RE_q$ is a birational model on which $\wtPsi\circ\wtPhi$ is algebraically stable. In order to see the action of $\wtPsi\circ\wtPhi$ on $N^1(\Bl_{R'}\Bl_RE_q)$, we compute the matrices that represent the action of $\wtPhi$ and $\wtPsi$ and then take their product.

\begin{prop}
The linear map $\wtPhi_*:N^1(\Bl_{R'}\Bl_RE_q)\to N^1(\Bl_{R'}\Bl_RE_q)$ is given by
\begin{equation*}
M_\Phi:=
\scriptsize{
\begin{bmatrix}
5 & 4 & 0 & 0 & 0 & 0 & 0 & 1 & 1 & 1 & 1 & 0 & 0 & 0 & 0 & 0 \\
0 & 0 & 0 & 0 & 0 & 0 & 1 & 0 & 0 & 0 & 0 & 0 & 0 & 0 & 0 & 0 \\
-1 & -1 & 1 & 0 & 0 & 0 & 0 & -1 & 0 & 0 & 0 & 0 & 0 & 0 & 0 & 0 \\
-1 & -1 & 0 & 1 & 0 & 0 & 0 & 0 & -1 & 0 & 0 & 0 & 0 & 0 & 0 & 0 \\
-1 & -1 & 0 & 0 & 1 & 0 & 0 & 0 & 0 & -1 & 0 & 0 & 0 & 0 & 0 & 0 \\
-1 & -1 & 0 & 0 & 0 & 1 & 0 & 0 & 0 & 0 & -1 & 0 & 0 & 0 & 0 & 0 \\
-4 & -3 & 0 & 0 & 0 & 0 & 0 & -1 & -1 & -1 & -1 & 0 & 0 & 0 & 0 & 0 \\
-2 & -2 & 0 & 0 & 0 & 0 & 0 & -1 & 0 & 0 & 0 & 0 & 0 & 0 & 0 & 0 \\
-2 & -2 & 0 & 0 & 0 & 0 & 0 & 0 & -1 & 0 & 0 & 0 & 0 & 0 & 0 & 0 \\
-2 & -2 & 0 & 0 & 0 & 0 & 0 & 0 & 0 & -1 & 0 & 0 & 0 & 0 & 0 & 0 \\
-2 & -2 & 0 & 0 & 0 & 0 & 0 & 0 & 0 & 0 & -1 & 0 & 0 & 0 & 0 & 0 \\
0 & 0 & 0 & 0 & 0 & 0 & 0 & 0 & 0 & 0 & 0 & 0 & 0 & 0 & 0 & 0 \\
-1 & -1 & 1 & 0 & 0 & 0 & 0 & -1 & 0 & 0 & 0 & 0 & 0 & 0 & 0 & 0 \\
-1 & -1 & 0 & 1 & 0 & 0 & 0 & 0 & -1 & 0 & 0 & 0 & 0 & 0 & 0 & 0 \\
-1 & -1 & 0 & 0 & 1 & 0 & 0 & 0 & 0 & -1 & 0 & 0 & 0 & 0 & 0 & 0 \\
-1 & -1 & 0 & 0 & 0 & 1 & 0 & 0 & 0 & 0 & -1 & 0 & 0 & 0 & 0 & 0
\end{bmatrix}
},
\end{equation*}
and the linear map $\wtPsi_*:N^1(\Bl_{R'}\Bl_RE_q)\to N^1(\Bl_{R'}\Bl_RE_q)$ is given by
\begin{equation*}
M_\Psi:=
\scriptsize{
\begin{bmatrix}
5 & 4 & 0 & 0 & 0 & 0 & 0 & 0 & 0 & 0 & 0 & 0 & 1 & 1 & 1 & 1 \\
0 & 0 & 0 & 0 & 0 & 0 & 0 & 0 & 0 & 0 & 0 & 1 & 0 & 0 & 0 & 0 \\
-1 & -1 & 1 & 0 & 0 & 0 & 0 & 0 & 0 & 0 & 0 & 0 & -1 & 0 & 0 & 0 \\
-1 & -1 & 0 & 1 & 0 & 0 & 0 & 0 & 0 & 0 & 0 & 0 & 0 & -1 & 0 & 0 \\
-1 & -1 & 0 & 0 & 1 & 0 & 0 & 0 & 0 & 0 & 0 & 0 & 0 & 0 & -1 & 0 \\
-1 & -1 & 0 & 0 & 0 & 1 & 0 & 0 & 0 & 0 & 0 & 0 & 0 & 0 & 0 & -1 \\
0 & 0 & 0 & 0 & 0 & 0 & 0 & 0 & 0 & 0 & 0 & 0 & 0 & 0 & 0 & 0 \\
-1 & -1 & 1 & 0 & 0 & 0 & 0 & 0 & 0 & 0 & 0 & 0 & -1 & 0 & 0 & 0 \\
-1 & -1 & 0 & 1 & 0 & 0 & 0 & 0 & 0 & 0 & 0 & 0 & 0 & -1 & 0 & 0 \\
-1 & -1 & 0 & 0 & 1 & 0 & 0 & 0 & 0 & 0 & 0 & 0 & 0 & 0 & -1 & 0 \\
-1 & -1 & 0 & 0 & 0 & 1 & 0 & 0 & 0 & 0 & 0 & 0 & 0 & 0 & 0 & -1 \\
-4 & -3 & 0 & 0 & 0 & 0 & 0 & 0 & 0 & 0 & 0 & 0 & -1 & -1 & -1 & -1 \\
-2 & -2 & 0 & 0 & 0 & 0 & 0 & 0 & 0 & 0 & 0 & 0 & -1 & 0 & 0 & 0 \\
-2 & -2 & 0 & 0 & 0 & 0 & 0 & 0 & 0 & 0 & 0 & 0 & 0 & -1 & 0 & 0 \\
-2 & -2 & 0 & 0 & 0 & 0 & 0 & 0 & 0 & 0 & 0 & 0 & 0 & 0 & -1 & 0 \\
-2 & -2 & 0 & 0 & 0 & 0 & 0 & 0 & 0 & 0 & 0 & 0 & 0 & 0 & 0 & -1
\end{bmatrix}
}.
\end{equation*}
\end{prop}
\begin{proof}
We will show the action of $\wtPhi$ is given by $M_\Phi$, and the proof for $\wtPsi$ is similar. 

First, we note that
\begin{equation*}
	\wtPhi_*(e_{p''})=\wtPhi_*(e_{p_i''})=0
\end{equation*}
for $i=1,\dots,4$.

Continuing the arguments in the proof of Proposition \ref{prop:actionofPhi}, we have
\begin{align*}
	\wtPhi_*(h)&=5h-4e_{p'}-\sum e_{p_i}-2\sum e_{p_i'}-\sum e_{p_i''}\\
	\wtPhi_*(e_p)&=4h-3e_{p'}-\sum e_{p_i}-2\sum e_{p_i'}-\sum e_{p_i''}\\
	\wtPhi_*(e_{p'})&=e_p\\
	\wtPhi_*(e_{p_i}+e_{p_i''})&=e_{p_i}+e_{p_i''}\\
	\wtPhi_*(e_{p_i'})&=h-e_{p'}-e_{p_i}-e_{p_i'}-e_{p_i''}.
\end{align*}
Also
\begin{equation*}
	\wtPhi_*(e_{p_i})=e_{p_i}+e_{p_i''}-\wtPhi_*(e_{p_i''})=e_{p_i}+e_{p_i''}.
\end{equation*}

\end{proof}

Proposition \ref{prop:Mgf=MgMf} suggests that we can use the product of $M_\Psi$ and $M_\Phi$ to represent the action of $\wtPsi\circ\wtPhi$ on $N^1(\Bl_{R'}\Bl_RE_q)$. The spectral radius of $M_\Psi M_\Phi$ is $16$. By Theorem \ref{thm:radius}, we obtain the following result:
\begin{prop}
The map $\Psi\circ\Phi$ on the exceptional divisor $E_q$ induced by $\chi\circ\varphi\circ\psi\circ\varphi$ has $\lambda(\Psi\circ\Phi)=16$.
\end{prop}

Thus, the conjugate map $E_\fq\dashrightarrow E_\fq$ induced by $\vartheta=\varphi\circ\chi\circ\varphi\circ\psi$ has $\lambda=16>1$. This proves the Theorem \ref{mainthm}, and we again found the interesting dynamics on the exceptional divisor $E_\fq$ in the fiber over $q$.

We also make some remarks about the regularizability of the map $\Psi\circ\Phi$. A birational map $f\in\Bir(\bp^2)$ is \textit{\textbf{regularizable}} if there exists a projective surface $S$ and a birational map $\pi:S\dashrightarrow\bp^2$ such that the lift $\pi^{-1}\circ f\circ\pi$ is an automorphism of $S$. It turns out that, for an automorphism of a projective surface, the dynamical degree, if different from $1$, is either a degree 2 algebraic integer or a Salem number (See \cite{DF2001,BC2016}). Here a \textit{\textbf{Salem number}} means an algebraic integer greater than $1$ whose other conjugates have absolute value $\leq 1$, with at least one having absolute value $1$. Since $16$ is neither a degree 2 algebraic integer nor a Salem number, the map $\Psi\circ\Phi$ cannot be birationally conjugate to an automorphism of a projective surface. Hence $\Psi\circ\Phi$ is not regularizable.

In a separate SageMath script, we compute an explicit equation for the map $\Psi\circ\Phi$. We show that the dynamical degree $\lambda=16$ given here is consistent with the direct approximations of the $\lambda$ given by assuming the Kawaguchi-Silverman conjecture and computing the height growth along a few orbits. Here we present a few terms in the approximation sequence along an orbit.
\begin{verbatim}
sage: P2.<x,y,z> = ProjectiveSpace(QQ,2)
sage: Phi = Hom(P2,P2)([
....: 9*x^5 - 10*x^3*y^2 + x*y^4 + 10*x^3*y*z - 2*x*y^3*z,
....: 9*x^4*y - 10*x^2*y^3 + y^5 + 9*x^4*z - y^4*z,
....: 9*x^4*z - 10*x^2*y^2*z + y^4*z])
sage: Psi = Hom(P2,P2)([
....: 9*x^5 - 10*x^3*y^2 + x*y^4 - 10*x^3*y*z + 2*x*y^3*z,
....: 9*x^4*y - 10*x^2*y^3 + y^5 - 9*x^4*z + y^4*z,
....: 9*x^4*z - 10*x^2*y^2*z + y^4*z])
sage: f = Psi*Phi
sage: f.normalize_coordinates()
sage: p = P2([4,5,1])
sage: l = []
sage: l.append(p)
sage: for i in range(1,6):
....:     l.append(f(p))
....:     p = f(p)
sage: l[-3].global_height()/l[-4].global_height()
15.9762820559196
sage: l[-2].global_height()/l[-3].global_height()
15.9989499436956
sage: l[-1].global_height()/l[-2].global_height()
15.9999881533067
\end{verbatim}

\bibliographystyle{amsplain}
\bibliography{Z_bibliography.bib}

\section*{Acknowledgement}

I extend my sincere gratitude to my advisor, John Lesieutre, for proposing this problem and for his invaluable discussions and guidance throughout this work. I thank the anonymous referees for their constructive feedback that enhanced the manuscript. This work was partially supported by NSF-DMS grant 2142966.

\end{document}